\documentclass[12pt]{amsart}

\usepackage{titlesec}
\usepackage[left=1.5in,right=1.5in,top=1.5in,bottom=1.5in]{geometry}
\usepackage{amsmath,amsfonts,amssymb,amsthm,mathtools}
\usepackage[stable]{footmisc}
\usepackage[backend=biber]{biblatex}
\usepackage{setspace}
\usepackage{hyperref}

\titleformat{\section}
  {\centering\normalfont\normalsize\bfseries}{\thesection.}{1em}{}
\titleformat{\subsection}
  {\normalfont\normalsize\bfseries}{\thesubsection.}{1em}{}
\titleformat{\subsubsection}
  {\normalfont\normalsize\bfseries}{\thesubsubsection.}{1em}{}

\newcommand{\Q}{\mathbf{Q}}
\newcommand{\R}{\mathbf{R}}
\newcommand{\N}{\text{N}}

\newcommand{\e}{\text{e}}
\newcommand{\diff}{\mathop{}\!\mathrm{d}}

\newcommand{\eps}{\varepsilon}

\newcommand{\wh}[1]{\widehat{#1}}
\newcommand{\wt}[1]{\widetilde{#1}}
\newcommand{\mc}[1]{\mathcal{#1}}
\newcommand{\mf}[1]{\mathfrak{#1}}

\newcommand{\con}[1]{\overline{#1}}

\newcommand{\smod}[1]{\;(#1)}

\renewcommand{\ge}{\geqslant}
\renewcommand{\le}{\leqslant}

\renewcommand{\Re}{\operatorname{Re}}

\renewcommand{\vec}[1]{\mathbf{#1}}

\newtheoremstyle{custom}
  {.5em}
  {.5em}
  {}
  {}
  {\bfseries}
  {:}
  {.5em}
  {}

\theoremstyle{custom}
\newtheorem{proposition}{Proposition}
\numberwithin{proposition}{section}
\newtheorem{lemma}[proposition]{Lemma}
\newtheorem{corollary}[proposition]{Corollary}

\newtheorem*{remark}{Remark}
\newtheorem*{remarks}{Remarks}
\newtheorem{theorem}[proposition]{Theorem}
\newtheorem{definition}[proposition]{Definition}

\title[Sifting for small split primes]{Sifting for small split primes \\ of an imaginary quadratic field \\ in a given ideal class}
\date{\today}
\author{Louis M. Gaudet}

\begin{document}


\begin{abstract}
Let $D>3$, $D\equiv3\smod{4}$ be a prime, and let $\mc{C}$ be an ideal class in the field $\Q(\sqrt{-D})$. In this article, we give a new proof that $p(D,\mc{C})$, the smallest norm of a split prime $\mf{p}\in\mc{C}$, satisfies $p(D,\mc{C})\ll D^L$ for some absolute constant $L$. Our proof is sieve theoretic. In particular, this allows us to avoid the use of log-free zero-density estimates (for class group $L$-functions) and the repulsion properties of exceptional zeros, two crucial inputs to previous proofs of this result.

\end{abstract}

\maketitle


\section{Introduction}

For integers $q\ge2$ and $(a,q)=1$, let $p(q,a)$ denote the least prime $p\equiv a\smod{q}$. In 1944, Linnik \cite{linnik1944leastI} showed that
\begin{equation}
\label{linnik theorem}
p(q,a)\;\ll\; q^L,
\end{equation}
where both $L$ and the implied constant are absolute. Since then, there have been many improvements on this result. Building on the work of Heath-Brown \cite{heath1992zero}, Xylouris \cite{xylouris2009linnik} showed that unconditionally one can take $L=5$ in \eqref{linnik theorem}, which is the current record. This comes quite close to the bound
\[
p(q,a)\;\ll\;(q\log q)^2,
\]
which is what follows assuming that the Riemann hypothesis holds for the Dirichlet $L$-functions $L(s,\chi)$. 

Such results are difficult to establish unconditionally, and have traditionally (following Linnik) depended on deep results on the zeros of these $L$-functions, namely a log-free zero-density estimate and a quantitative version of the Deuring-Heilbronn phenomenon (exceptional zero repulsion effect). 

Linnik's theorem has been generalized in the setting of the Chebotarev density theorem: given a Galois extension of number fields $L/K$ with Galois group $G$, each prime $\mf{p}$ of $K$ (unramified in $L$) can be associated to a conjugacy class $C\subset G$ by the Artin symbol $\big[\frac{L/K}{\mf{p}}\big]$. The analogue of Linnik's theorem is a bound (in terms of the various number field parameters involved) on the least norm $\N_{K/\Q}\mf{p}$ of a prime $\mf{p}$ with prescribed Artin symbol. There have been many works in this direction, both conditional (see \cite{lagarias1977effective} and \cite{bach1996explicit}) and unconditional (see \cite{fogels1962distribution}, \cite{montgomery1979bound}, \cite{weiss1983least}, \cite{kowalski2002zeros}, \cite{zaman2017bounding}, and \cite{thorner2017explicit}). These unconditional results proceed by establishing analogues of Linnik's log-free zero-density estimate and quantitative Deuring-Heilbronn phenomenon for the Hecke $L$-functions.

In a different direction, there has been a growing interest in finding new proofs of Linnik's theorem that avoid using input about the zeros of $L$-functions---see for instance \cite{elliot2002least}, \cite{granville2019new}, \cite{koukoulopoulos2020distribution}, \cite{sachpazis2023pretentious}, \cite{friedlander2023selberg}, \cite{friedlander2023sifting}, \cite{matomaki2024products}, \cite{matomaki2024primes}, and Chapter 24 in \cite{friedlander2010opera}. Often these works combine sieve theoretic techniques with ``pretentious methods'' and/or with techniques coming from additive combinatorics. While such proofs are not usually (as of yet) as numerically strong as those that use the zeros of $L$-functions, they are very interesting from a conceptual point of view. Similar to the ``elementary proof'' of the prime number theorem, such proofs show how challenging results in arithmetic can be achieved without (or with minimal) use of results on the zeros of $L$-functions.

In this article, we are interested in a particular analogue of Linnik's theorem for primes in imaginary quadratic fields, which we prove without using zero-density theorems or exceptional zero repulsion results. Let $D>3$ be a prime, $D\equiv3\smod{4}$, so that $-D$ is a negative fundamental discriminant. (We work with prime $D$ for simplicity, though we expect that our methods could be adapted without major modifications to work for general fundamental discriminants $-D$.) For an integral ideal $\mf{a}\subseteq\mc{O}_K$, we denote by $\N\mf{a}=|\mc{O}_K/\mf{a}|$ its (absolute) norm. The class group $\mc{H}$ of $K$ is a finite abelian group of order $h=|\mc{H}|$, the class number.

In analogy to Dirichlet's theorem on primes in arithmetic progression, one can show using class group characters $\chi\in\wh{\mc{H}}$ that for any specified ideal class $\mc{C}\in\mc{H}$, there are infinitely many split primes $\mf{p}$ (i.e., unramified primes $\mf{p}$ with $\N\mf{p}=p$, a rational prime) in the class $\mc{C}$. Therefore, in analogy to Linnik's theorem, we are interested in bounding
\[
p(D,\mc{C})\;\coloneqq\;\min\{\N\mf{p}:\mf{p}\in\mc{C}\text{ is a split prime of }\mc{O}_K\},
\]
the least norm of a split prime in the class $\mc{C}$. Our main result is

\begin{theorem}
\label{main theorem}
There is an absolute constant $L>0$ such that
\[
p(D,\mc{C})\;\ll\;D^L,
\]
and the implied constant is absolute.
\end{theorem}

This result is a special case of the results for the Chebotarev theorem, so it has been established several times before in those works listed above. In particular, Thorner and Zaman \cite{thorner2017explicit} showed
\[
p(D,\mc{C}) \;\ll\; D^{694}
\]
for all negative fundamental discriminants $-D$. However, our work is novel in that it is the first time such a result has been established without the use of zero-density theorems and quantitative Deuring-Heilbronn results. We also handle new sieve-theoretic challenges (in comparison with works on Linnik's theorem) in the sifting dimension aspect; see Section \ref{sec:outline} for details.

\begin{remarks}
In Zaman's thesis \cite{zaman2017analytic} it is shown that
\[
p(D,\mc{C})\;\ll\;D^{455},
\]
which is the current record to our knowledge.

Ditchen \cite{ditchen2013primes} has shown that except for a density zero subset of negative fundamental discriminants $-D\not\equiv0\smod{8}$, one has $p(D,\mc{C})\ll D^{20/3+\eps}$. For this result, they establish a large sieve inequality for class group characters on average over discriminants, as well as an analogue of the Bombieri-Vinogradov theorem for primes in ideal classes.

Given the correspondence between imaginary quadratic fields and binary quadratic forms (see \cite{cox2022primes}, for instance), for the principal class $\mc{C}_0$, the quantity $p(D,\mc{C}_0)$ is the also the least prime of the form $p=x^2+Dy^2$ when $D\equiv1\smod{4}$. In our case with $D\equiv3\smod{4}$, $p(D,\mc{C}_0)$ is the least prime $p$ of the form $4p=x^2+Dy^2$. The distribution of such primes has been studied by Fouvry and Iwaniec \cite{fouvry2003low-lying} in connection with low-lying zeros of dihedral $L$-functions.
\end{remarks}

\section{Statement of results}

Theorem \ref{main theorem} is the result of combining Theorems \ref{exceptional theorem} and \ref{nonexceptional theorem} below; we now establish our notations and state these theorems.

Given an ideal class $\mc{C}$ in the class group $\mc{H}$ of $K=\Q(\sqrt{-D})$, we put
\[
\lambda_{\mc{C}}(n)\;=\;\#\{\mf{a}\in\mc{C};\; \N\mf{a}=n\}.
\]
Given a character $\chi\in\wh{\mc{H}}$ of the class group, we define
\begin{equation}
\label{definition of lambda chi}
\lambda_\chi(n)\;=\;\sum_{\N\mf{a}=n}\chi(\mf{a}),
\end{equation}
the sum being taken over integral ideals $\mf{a}\subseteq\mc{O}_K$ of norm $n$. Then by the orthogonality of the class group characters we have
\begin{equation}
\label{orthogonality of characters}
\lambda_{\mc{C}}(n)\;=\;\frac{1}{h}\sum_{\chi\in\wh{\mc{H}}}\con{\chi}(\mc{C})\lambda_\chi(n).
\end{equation}
For the trivial character $\chi_0\in\wh{\mc{H}}$ we have
\begin{equation}
\label{trivial character formula}
\lambda_{\chi_0}(n)\;=\;(1*\chi_D)(n)\;=\;\sum_{d\mid n}\chi_D(d),
\end{equation}
which is the number of ideals in $\mc{O}_K$ of norm $n$, and
\[
\chi_D(n)\;=\;\Big(\frac{-D}{n}\Big)
\]
is the Kronecker symbol. Indeed, $\chi_D$ is a primitive real Dirichlet character of conductor $D$. By the Dirichlet class number formula, we have
\[
h\;=\;\frac{1}{\pi}\sqrt{D}\;L(1,\chi_D).
\]
Our main object of study is the sequence
\begin{equation}
\label{definition of a_n}
a_n\;=\;\frac{1}{n}\lambda_\mc{C}(n)f\Big(\frac{\log n}{\log x}\Big),
\end{equation}
$f(u)\ge0$ a smooth function with $\wh{f}(0)>0$, $\wh{f}$ denoting the Fourier transform,
\[
\wh{f}(\xi)\;\coloneqq\;\int_{-\infty}^\infty f(u)\e(-\xi u)\diff u,
\]
and $\e(z)\coloneqq e^{2\pi iz}$. We assume that $f(u)$ is supported in the segment
\[
1-\nu\le u\le1,
\]
where $\nu>0$ is a small concrete number whose value can be determined in the course of our arguments (though the exact value is not important to us). Here, $x$ is a large parameter going to infinity, and our goal is to estimate
\[
S\;\coloneqq\;\sum_pa_p\;=\;\sum_p\frac{1}{p}\lambda_\mc{C}(p)f\Big(\frac{\log p}{\log x}\Big).
\]
By the prime number theorem, we have
\[
\sum_p\frac{1}{p}f\Big(\frac{\log p}{\log x}\Big)\;\sim\;\wt{f}(0)\;>\;0,
\]
where $\wt{f}$ denotes the Mellin transform of $f$,
\[
\wt{f}(s)\;=\;\int_0^\infty u^{s-1}f(u)\diff u.
\]
One expects prime ideals $\mf{p}$ to equidistribute among the $h$ ideal classes $\mc{C}$ in $\mc{H}$ even when the discriminant $D$ is comparable in size (in the logarithmic scale) to the norm $\N\mf{p}$. Thus we expect the asymptotic formula
\[
S\;\sim\;\frac{1}{h}\;\wt{f}(0)
\]
to hold for $x\ge D^A$ for some absolute constant $A>0$. Indeed, the Riemann hypothesis for the class group $L$-functions $L_K(s,\chi)$ implies that the above asymptotic formula holds with $x\ge D^2(\log D)^4$. Here we establish the bound
\begin{equation}
\label{bounds to establish}
S \;\gg\; \frac{1}{h}\;\wt{f}(0)
\end{equation}
uniformly for $x\gg D^L$ for some absolute $L>0$. (In actuality, we prove a slightly weaker lower bound in the case of an exceptional character; see the precise statement in Theorem \ref{exceptional theorem} and the remarks that follow.) In particular, it follows from this lower bound that for all prime $D\equiv3\smod{4}$, we have
\begin{equation}
\label{linnik type result}
\lambda_\mc{C}(p)>0\quad\text{for some }p\ll D^L. 
\end{equation}
In other words, every class $\mc{C}$ contains a prime ideal $\mf{p}$ with $p=\N\mf{p}\ll D^L$.

To establish the bound \eqref{bounds to establish} unconditionally, we split our argument into two cases depending on the non/existence of real zeros of the Dirichlet $L$-function $L(s,\chi_D)$. We will use assumptions about such zeros in several places in the work (and also for other $L$-functions), so to clarify this, we make the following

\begin{definition}
Let $L(s,f)$ be an $L$-function of conductor $\Delta\ge3$ (see Appendix \ref{section euler products} for definitions). For a real number $c>0$, we say that ``Hypothesis $\text{H}(c)$'' holds for $L(s,f)$ if every zero $\rho=\beta+i\gamma$ of $L(s,f)$ with $|\gamma|\le1$ satisfies
\begin{equation}
\label{assumption on zeros}
\beta\;\le\;1-\frac{c}{\log \Delta}.
\end{equation}
\end{definition}

Now we state our main two theorems, which together prove Theorem \ref{main theorem}.

\begin{theorem}
\label{exceptional theorem}
There exists an absolute constant $c>0$ such that if the Dirichlet $L$-function $L(s,\chi_{D})$ has a real zero $\beta$ that satisfies
\[
\beta\;>\;1-\frac{c}{\log D},
\]
then we have
\begin{equation}
\label{exceptional S bound}
S\;\gg\;\frac{\wh{f}(0)}{h}\frac{L(1,\chi_D)}{\log x}.
\end{equation}
\end{theorem}

\begin{theorem}
\label{nonexceptional theorem}
Let $c>0$ be the constant from the theorem above, and suppose that Hypothesis $\text{H}(c)$ holds for $L(s,\chi_D)$. Then we have
\begin{equation}
\label{nonexceptional S bound}
S\;\gg\;\frac{\nu}{h}.
\end{equation}
\end{theorem}

\begin{remarks}
In both theorems above, the implied constants are absolute and effective, though we make no attempt at computing them here. See \cite{gaudet2023least}, where they compute an explicit admissible value for the exponent $L$ in \eqref{linnik type result}.

We have $\wt{f}(0)\ll\nu$, so the bound \eqref{nonexceptional S bound} implies \eqref{bounds to establish}. On the other hand, the lower bound in \eqref{exceptional S bound} is somewhat smaller than the true order of magnitude. This is because in the proof we have used the crude bound $V(z)\gg(\log x)^{-2}$ that does not involve cancellation in sums over $\chi_D(p)$. We would recover the correct order of magnitude $XV(z)\gg\wh{f}(0)$ if we showed that $L(1,\chi_D)$ is well-approximated by the product $\prod_{p<z}(1-\chi_D(p)/p)^{-1}$. This is of course expected to be the case, and we prove it in Appendix \ref{section euler products} under the complementary assumption that $\chi_D$ is not exceptional. In any case, the bound \eqref{exceptional S bound} still shows that there are many primes $p\ll D^L$ with $\lambda_{\mc{C}}(p)>0$. 
\end{remarks}

\section{Outline of the arguments}
\label{sec:outline}

We follow the general approach of Friedlander and Iwaniec's proof of Linnik's theorem in Chapter 24 of \cite{friedlander2010opera}---see also their recent related articles \cite{friedlander2023selberg} and \cite{friedlander2023sifting}. Our work differs from theirs in a few key aspects, which we now explain.

The goal is to give a positive lower bound for the sum $\sum_pa_p$, where for them $\mc{A}=(a_n)$ is the characteristic function of the arithmetic progression $n\equiv a\smod{q}$, $x/2<n\le x$, and for us $(a_n)$ is defined by \eqref{definition of a_n}. They begin with an application of Buchstab's identity,
\begin{equation}
\label{outline:buchstab}
S(\mc{A},\sqrt{x})\;=\;S(\mc{A},z)\;-\sum_{z\le p<\sqrt{x}}S(\mc{A}_p,p),
\end{equation}
where $S(\mc{A},w)$ denotes the sum of $(a_n)$ over $n$ having no prime factor less than $w$, and $\mc{A}_p=(a_{pm})$ denotes the subsequence of $\mc{A}$ over multiples of $p$. Here we take $z=x^{1/r}$ with $r$ taken to be as large as necessary.

First they treat the case where there is an exceptional character $\chi\smod{q}$. In this case, they apply \eqref{outline:buchstab} to the ``twisted'' sequence $\wt{a}_n=\lambda(n)a_n$, where
\[
\lambda(n)=\sum_{d\mid n}\chi(d).
\]
They show using the Fundamental Lemma of Sieve Theory that the two terms on the right-hand side of \eqref{outline:buchstab} are (asymptotically as $r$ becomes large) of the same size, up to a factor of
\[
\delta(z,x)\;\coloneqq\;\sum_{z\le p<x}\frac{\lambda(p)}{p}\;=\;\sum_{z \le p<x}\frac{1+\chi(p)}{p}
\]
present in the second term. Assuming that $\chi$ is exceptional, $\delta(z,x)$ is very small, and a positive lower bound for $S(\mc{A},\sqrt{x})=\sum_pa_p$ follows. Friedlander and Iwaniec also give in \cite{friedlander2023selberg} an alternative approach via Selberg's sieve that works on similar principles and gives comparable results. 

We follow their approach in \cite{friedlander2010opera} to prove our Theorem \ref{exceptional theorem}, which is under the assumption of an exceptional character. This is taken up in Section \ref{sec:exceptional}. We require little modification of their arguments, since their method does not require very specific properties of the sequence $\mc{A}=(a_n)$ beyond some basic sieve assumptions that also apply in our case. In fact, it is even simpler for us, since we have no need to twist our sequence $(a_n)$ by the weights $\lambda(n)$ above---such a factor naturally appears in this particular sequence already; see Proposition \ref{lambda chi with sieve weights} for a precise statement.

For the non-exceptional case, Friedlander and Iwaniec work with a combinatorial sieve identity that leads to
\begin{equation}
\label{fi sieve inequality}
S(\mc{A},z)\;\ge\;S^-(\mc{A},z)\;+\;\frac{1}{24}Q(\mc{A}),
\end{equation}
where
\[
S^-(\mc{A},z)\;=\;\sum_{d\mid P(z)}\lambda_d^-A_d
\]
is the lower bound coming from the beta-sieve (so that $(\lambda_d^-)$ are the lower-bound beta-sieve weights, and $A_d$ are the congruence sums for the sequence $\mc{A}$---see Section \ref{sec:beta-sieve} for details), and
\[
Q(\mc{A})\;\coloneqq\;\sum_{p_0}\sum_{p_1}\sum_{p_2}\sum_{p_3}\sum_{p_4}a_{p_0p_1p_2p_3p_4},
\]
where the variables $p_j$ run over specific segments $x^{\alpha_j}\le p_j\le x^{\beta_j}$. 

Their sifting problem is linear (i.e. of sieve dimension $\kappa=1$; again, see \ref{sec:beta-sieve} for details), which means that one can show that $S^-(\mc{A},z)$ is negligible (relatively very small) for $z$ close to $\sqrt{x}$. The upshot is that they show that
\[
S(\mc{A},\sqrt{x})\;\gtrsim\;\frac{1}{24}Q(\mc{A})
\]
up to some comparably negligible contributions. This reduces the problem of counting primes to finding a lower bound for $Q(\mc{A})$, which counts products of prime quintuplets in arithmetic progression. Indeed, the common parity here (products of 1 and 5 primes, respectively) is an artifact of the sieve process. 

By contrast, in this work we cannot so readily work with \eqref{fi sieve inequality}. The sifting density function $g(d)$ for our sequence $\mc{A}=(a_n)$ is given on primes $p$ by
\[
g(p)\;=\;\frac{1+\chi_D(p)}{p}\;+\;O\Big(\frac{1}{p^2}\Big),
\]
and the presence of the character $\chi_D(p)$ causes fluctuations that hinder one from easily claiming the one-dimensionality of the sieve problem. One possible approach would be to use the fact that we are working in the non-exceptional case (i.e., assuming Hypothesis $\text{H}(c)$ for $L(s,\chi_D)$, say) to effectively bound the sum $\sum_{p<z}g(p)$ and hence control the sieve dimension.

However, here we choose to proceed differently: by the trivial bound $g(p)\le2/p+O(1/p^2)$, we can work with a $\kappa=2$-dimensional sieve. We can no longer show that $S^-(\mc{A},z)$ is negligible for $z$ so close to $\sqrt{x}$ (only for $z\le x^{1/\beta(2)}= x^{1/4.8339\dots}$; see Section \ref{sec:beta-sieve}), and so we employ a different combinatorial identity than in \cite{friedlander2023sifting}. This identity comes from applying a second iteration of the Buchstab formula to each term $S(\mc{A}_p,p)$ in \eqref{outline:buchstab}, which gives
\begin{equation}
\label{outline:buchstab 2}
S(\mc{A},\sqrt{x}) \;=\; S(\mc{A},z) - \sum_{z\le p<\sqrt{x}}S(\mc{A}_{p},z)+\mathop{\sum\sum}_{z\le p_2<p_1<\sqrt{x}}S(\mc{A}_{p_1p_2},p_2).
\end{equation}
Rather than work with an inequality, we evaluate (nearly asymptotically) each of the three terms on the right-hand side of \eqref{outline:buchstab 2} and show that the result is positive. The first two terms are readily handled via the Fundamental Lemma, so we reduce the problem to analyzing the third term, which is
\begin{equation}
\label{our Q(A)}
\mathop{\sum\sum}_{z\le p_2<p_1<\sqrt{x}}S(\mc{A}_{p_1p_2},p_2) \;= \mathop{\sum\sum\sum}_{\substack{z\le p_2<p_1<\sqrt{x} \\ (b,P(p_2))=1}}a_{p_1p_2b}.
\end{equation}
This sum is our analogue of $Q(\mc{A})$---note that it is supported on integers which are products of three (almost-) primes, the same parity as in $Q(\mc{A})$. 

Friedlander and Iwaniec handle $Q(\mc{A})$ via a multiplicative analogue of a additive ternary problem treated by the classical circle method. They use Dirichlet characters $\chi\smod{q}$ to decouple the prime variables $p_j$. After removing the contribution of the principal character (the ``major arc''), they use two of the prime variables and the orthogonality of the characters to recover the cost of opening the sum with the characters. The remaining three prime variables are used to obtain a nontrivial cancellation in the character sums over primes---importantly, they do not have need for any zero density bounds or repulsion properties of the exceptional zeros. 

We handle the sum \eqref{our Q(A)} in a similar manner, here using the class group characters $\chi\in\wh{\mc{H}}$ instead of Dirichlet characters to decouple our variables. Just as above, two variables ($p_1$ and $b$) and the orthogonality of the class group characters are used to recover the cost in using these characters. This involves a type of large sieve inequality for these characters (over integers free from small prime factors) that we develop in Section \ref{sec:ls inequality}. For nontrivial cancellation in a character sum over the final prime variable $p_2$, we apply the explicit formula and use a zero-free region for the class group $L$-functions. It is a technical reason that we do not use three prime variables for this as they do in \cite{friedlander2023sifting}. While their sequence $(a_n)$ is localized dyadically, $x/2<n\le x$, ours is supported in a longer segment $x^{1-\nu}\le n\le x$. This means that we work with a longer sum over the prime variable $p_2$, which effectively localizes the dual sum over zeros in the explicit formula to essentially be supported on zeros within the classical zero-free region. It is in this way that we do not make use of any zero density estimates or repulsion effects of exceptional zeros.

\section{Acknowledgments}

This work was completed as part of the author's PhD thesis. He is deeply grateful to his advisor, Henryk Iwaniec, who provided constant support and guidance throughout this project, and who provided very insightful and helpful feedback during the writing of this article.

\section{Preliminaries}

\subsection{The beta-sieve}
\label{sec:beta-sieve}

For a nonnegative sequence of real numbers $\mc{A}=(a_n)$ we define
\[
S(\mc{A},z)\;\coloneqq\sum_{(n,P(z))=1}a_n
\]
for $z\ge2$, where $P(z)\;=\;\prod_{p<z}p$. The congruence sums for $\mc{A}$ are
\[
A_d\;\coloneqq \sum_{n\equiv0\smod{d}}a_n,
\]
which we will evaluate in the form
\begin{equation}
\label{congruence sum approximation}
A_d\;=\;g(d)X\;+\;r_d,
\end{equation}
where $g(d)$ is a multiplicative function with $0\;\le\; g(p)\;<\;1$ for prime $p$, $X$ is a smooth approximation to $A_1$, and $r_d$ is a remainder term that is small (on average over $d$) in comparison to $g(d)X$. The range of the modulus $d$ for which \eqref{congruence sum approximation} holds is called the level of distribution of the sequence $\mc{A}$. 

By the inclusion-exclusion principle, one expects that
\begin{equation}
\label{sieve expectation}
S(\mc{A},z)\;\asymp\; XV(z),
\end{equation}
where
\[
V(z)\;\coloneqq\; \prod_{p<z}(1-g(p)). 
\]
To establish the estimates \eqref{sieve expectation}, we use a sequence of sieve weights $\xi=(\xi_d)$, which are real numbers $\xi_d$ supported on squarefree integers $d$ satisfying
\[
d\mid P(z),\qquad d\le y,
\]
and we call $y$ the level of the sieve. We assume that they satisfy
\begin{equation}
\label{sieve weights bounded}
|\xi_d|\;\le\;1\quad\text{for all }d. 
\end{equation}
For sieve weights $(\xi_d)$, we put $\theta\;=\;1*\xi$; that is,
\[
\theta=(\theta_n),\qquad\theta_n\;=\;\sum_{d\mid n}\xi_d.
\]
To achieve lower- and upper-bounds as in \eqref{sieve expectation}, we use two sets of weights $(\xi_d^-)$ and $(\xi_d^+)$, called lower- and upper-bound sieve weights. We put
\begin{equation}
\label{sifted sum definition}
S^\pm(\mc{A},z)\;=\;\sum_na_n\theta_n^\pm,\quad\text{where}\quad\theta^\pm\;=\;1*\xi^\pm,
\end{equation}
and we require that
\[
\theta_n^-\;\le\;\sum_{d\mid(n,P(z))}\mu(d)\;\le\;\theta_n^+\qquad\text{for all }n,
\]
which implies that
\[
S^-(\mc{A},z)\;\le\; S(\mc{A},z)\;\le\; S^+(\mc{A},z).
\]
Finally, we say that our sifting problem has dimension at most $\kappa\ge0$ if
\begin{equation}
\label{sieve dimension condition}
\prod_{w\le p<z}(1-g(p))^{-1}\;\le\; K\Big(\frac{\log z}{\log w}\Big)^\kappa
\end{equation}
for every $2\le w<z$, for some constant $K>1$.

While many choices of sieve weights would suffice for our purposes (any that furnish a strong-enough ``fundamental lemma'' result), for concreteness we will from here on work with a specific construction of sieve weights known as the beta-sieve. These weights were first constructed by Iwaniec \cite{iwaniec1980rosser} and also appear in unpublished work of Rosser. They are of combinatorial type, and they satisfy all of the general properties discussed above, including \eqref{sieve weights bounded}---see Chapter 11 in \cite{friedlander2010opera} for a comprehensive treatment.

The main result we require about the beta-sieve weights is

\begin{proposition}[see Theorem 11.13 in \cite{friedlander2010opera}]
\label{fundamental lemma}
Let $\xi^\pm$ be the upper- and lower-bound beta-sieve weights of level $y$. Let $\mc{A}=(a_n)$ be a sequence of nonnegative reals, let $r_d$ be defined by \eqref{congruence sum approximation}, and assume that $g(d)$ satisfies \eqref{sieve dimension condition} with $\kappa\ge0$.

Let $z\ge2$ and put $s=\log y/\log z$. Define $S^\pm(\mc{A},z)$ by \eqref{sifted sum definition}, and put
\[
R^\pm(y,z)\;=\;\sum_{d\mid P(z)}\xi_d^\pm r_d.
\]
Then we have
\[
S^+(\mc{A},z) \;\le\; XV(z)\Big\{\mf{F}(s)+O((\log y)^{-1/6})\Big\}\;+\;R^+(y,z)
\]
for $s\ge\beta-1$, and
\[
S^-(\mc{A},z) \;\ge\; XV(z)\Big\{\mf{f}(s)+O((\log y)^{-1/6})\Big\}\;+\;R^-(y,z)
\]
for $s\ge\beta$, where $\beta=\beta(\kappa)$ is a specific absolute constant that depends only on the sifting dimension $\kappa$, $\mf{F}(s)$ and $\mf{f}(s)$ are the continuous solutions to the following system of differential-difference equations,
\begin{align*}
&\begin{cases}
s^\kappa\mf{F}(s)=A & \text{if }\beta-1\le s\le\beta+1, \\
s^\kappa\mf{f}(s)=B & \text{at }s=\beta,
\end{cases} \nonumber \\
&\begin{cases}
(s^\kappa\mf{F}(s))' = \kappa s^{\kappa-1}\mf{f}(s-1) & \text{if }s>\beta-1, \\
(s^\kappa\mf{f}(s))' = \kappa s^{\kappa-1}\mf{F}(s-1) & \text{if }s>\beta,
\end{cases}
\end{align*}
and $A=A(\kappa)$ and $B=B(\kappa)$ are specific absolute constants that depend only on the sifting dimension $\kappa$. As $s\to+\infty$, we have
\[
\mf{F}(s)\;=\;1+O(e^{-s})\qquad\text{and}\qquad\mf{f}(s)\;=\;1+O(e^{-s}).
\]
\end{proposition}

We will only apply Proposition \ref{fundamental lemma} in the case $\kappa=2$. In this case we have $B=0$; in fact, $B(\kappa)=0$ if $\kappa\ge1/2$. On the other hand, it is nontrivial to compute the values of $\beta(\kappa)$ and $A(\kappa)$ for $k\ge1/2$; see Chapter \S11.19 in \cite{friedlander2010opera} for a discussion of this and a number of useful inequalities. In particular, they provide a table of numerical values of $\beta$ and $A$ for specific values of $\kappa\ge1/2$ that were computed by Sara Blight in 2009 and confirmed by Alastair J. Irving in 2014 (who also corrected one value in the table). When $\kappa=2$, we have
\[
\beta(2)\;=\;4.8339865967\dots\qquad\text{and}\qquad A(2)\;=\;43.4968874616\dots. 
\]

\begin{remark}
Since $|\xi_d^\pm|\le1$ for the beta-sieve weights, $R^\pm$ are bounded by
\[
R(y,z)\;=\sum_{\substack{d\mid P(z) \\ d\le y}}|r_d|.
\]
We will always bound the sieve remainder terms absolutely in this work; we have no need to extract additional cancellation from among these terms.
\end{remark}

\subsection{Class group $L$-functions}

Given a character $\chi$ of the class group $\mc{H}$, we define the associated $L$-function by
\[
L_K(s,\chi)\;=\;\sum_{\mf{a}}\chi(\mf{a})(\N\mf{a})^{-s} \;=\; \sum_{n\ge1}\lambda_\chi(n)n^{-s},
\]
the first sum being taken over all nonzero integral ideals $\mf{a}$ of $\mc{O}_K$. These functions are entire except in the case that the character is the trivial one, $\chi=\chi_0$; in this case, the above $L$-function is the Dedekind zeta function associated to the field $K$, 
\[
L_K(s,\chi_0)\;=\;\zeta_K(s)\;=\;\zeta(s)L(s,\chi_D),
\]
where $L(s,\chi_D)$ is the Dirichlet $L$-function associated to the Kronecker symbol $\chi_D$. Note that $\chi_D$ is primitive, since $-D$ is a fundamental discriminant.

The functions
\[
f_\chi(z)\;=\;h\delta(\chi)+\sum_{n\ge1}\lambda_\chi(n)\e(nz)
\]
are modular forms of weight 1 for the group $\Gamma_0(D)$ with nebentypus $\chi_D$. When $\chi\ne\chi_0$, they are Hecke eigencuspforms, and in fact they are newforms because the character $\chi_D$ is primitive. Thus it follows that the coefficients satisfy the Hecke relations
\begin{equation}
\label{Hecke relations}
\lambda_\chi(dm)\;=\sum_{q\mid(d,m)}\mu(q)\chi_D(q)\lambda_\chi\Big(\frac{d}{q}\Big)\lambda_\chi\Big(\frac{m}{q}\Big)\qquad\text{for all integers }d,m\ge1.
\end{equation}
A convenient reference for these facts is \cite{iwaniec1997topics}; see in particular \S6.6. 
The class group $L$-functions are self-dual in the sense that
\[
L_K(s,\con{\chi})=L_K(s,\chi),
\]
since for all $n\ge1$ we have
\[
\con{\lambda_\chi}(n)=\lambda_\chi(n)
\]
even though the character $\chi\in\wh{\mc{H}}$ need not be real. 
The completed $L$-functions
\[
\Lambda_K(s,\chi)=\gamma(s)D^{s/2}L_K(s,\chi),\qquad \gamma(s)\coloneqq(2\pi)^{-s}\Gamma(s),
\]
satisfy the functional equation (with root number $\eps=1$)
\begin{equation}
\label{functional equation}
\Lambda_K(s,\chi)=\Lambda_K(1-s,\chi).
\end{equation}

\subsection{The explicit formula and zeros of $L_K(s,\chi)$}
\label{chapter zeros}

By logarithmic differentiation of the Euler products
\[
L_K(s,\chi)=\prod_\mf{p}(1-\chi(\mf{p})(\N\mf{p})^{-s})^{-1}=\prod_p(1-\lambda_\chi(p)p^{-s}+\chi_D(p)p^{-2s})^{-1},
\]
we get
\[
-\frac{L_K'}{L_K}(s,\chi)=\sum_{\mf{a}}\Lambda_\chi(\mf{a})(\N\mf{a})^{-s} = \sum_{n\ge1}\Lambda_\chi(n)n^{-s},
\]
where (by a slight abuse of notation)
\[
\Lambda_\chi(\mf{a}) =
\begin{cases}
\chi(\mf{a})\log\N\mf{p} & \text{ if }\mf{a}=\mf{p}^k,\\
0 & \text{otherwise,}
\end{cases}
\qquad \text{and} \qquad
\Lambda_\chi(n) = \sum_{\N\mf{a}=n}\Lambda_\chi(\mf{a}).
\]
Thus $\Lambda_\chi(\mf{a})$ is supported on powers of prime ideals $\mf{p}$, and $\Lambda_\chi(n)$ is supported on powers of (rational) primes $p$. Note that for a rational prime $p$ we have
\[
\Lambda_\chi(p)=\lambda_\chi(p)\log p.
\] 
Using standard arguments (see Theorem 5.11 in \cite{iwaniec2004analytic}, for instance), we have

\begin{lemma}
Let $\Phi$ be a smooth, compactly supported function on $\R^+$ with Mellin transform $\wt{\Phi}$. Then we have
\begin{align}
\sum_{n\ge1}\Lambda_\chi(n)&\Phi(n) \;=\; \wt{\Phi}(1)\delta(\chi) - \sum_\rho\wt{\Phi}(\rho) + \Phi(1)\log D \nonumber \\
\label{explicit formula}
&+\sum_{n\ge1}\frac{\Lambda_\chi(n)}{n}\Phi\Big(\frac{1}{n}\Big)+\frac{1}{2\pi i}\int_{(3/2)}\wt{\Phi}(1-s)\Big(\frac{\gamma'}{\gamma}(s)+\frac{\gamma'}{\gamma}(1-s)\Big)\diff s,
\end{align}
where the sum over $\rho$ is taken over all nontrivial zeros of $L_K(s,\chi)$.
\end{lemma}

We will apply this formula with test functions $\Phi$ that have the form
\begin{equation}
\label{from phi to Phi}
\Phi(t)=\frac{1}{t\log t}\phi\Big(\frac{\log t}{\log x}\Big),
\end{equation}
where $\phi(u)$ is a compactly supported function on $\R^+$. In this case we have
\begin{proposition}
\label{proposition explicit formula in log scale}
Let $\phi$ be a smooth function supported on $[\alpha_1,\alpha_2]\subset[0,1]$. Putting $\theta=\log D/\log x$, suppose that $\theta<2\alpha_1-\alpha_2$. Then we have
\begin{equation}
\label{explicit formula in log scale}
\sum_p\frac{\lambda_\chi(p)}{p}\phi\Big(\frac{\log p}{\log x}\Big)\;=\;\wt{\Phi}(1)\delta(\chi)\;-\;\sum_\rho\wt{\Phi}(\rho)\;+\;O\Big(\frac{1}{h\log x}\Big),
\end{equation}
where $h=h(D)$ is the class number, $\Phi$ is defined as in \eqref{from phi to Phi}, and $\rho$ runs over all nontrivial zeros of $L_K(s,\chi)$. The implied constant depends only on $\phi$. 
\end{proposition}

\begin{proof}
Since $\phi(u)=0$ for $u<\alpha_1$, we see that $\Phi(1/n)=0$ for all $n\ge1$, and so the third and fourth terms on the right-hand side of \eqref{explicit formula} vanish. The rest of the proof follows in a standard way: integrate by parts to estimate the integral, and estimate trivially (using the assumption $\theta<2\alpha_1-\alpha_2$) the contribution of prime powers to the left-hand side of \eqref{explicit formula}.
\end{proof}

Finally, we record the following result that we will use in Subsection \ref{Evaluating V} to estimate a sum over the zeros $\rho$ of $L_K(s,\chi)$. 

\begin{proposition}
\label{prop hyp HC for class group l func}
Suppose that the Dirichlet $L$-function $L(s,\chi_D)$ satisfies Hypothesis $\text{H}(c)$ with $c\le1/12$. Then each of the class group $L$-functions $L_K(s,\chi)$ satisfy Hypothesis $\text{H}(c)$ as well.
\end{proposition}

\begin{proof}
Zaman \cite{zaman2016explicit} has shown that for $D$ sufficiently large, the product of $L$-functions $\prod_{\chi\in\mc{H}}L_K(s,\chi)$ has at most one zero in the region
\begin{equation}
\label{zaman zfr}
\sigma\;\ge\;1-\frac{0.0875}{\log D+3},\qquad|t|\le1,
\end{equation}
and that if such a zero exists, then both it and the associated class group character $\chi$ are real. By the genus theory, the only real class group character $\chi\in\mc{H}$ is the trivial character $\chi=\chi_0$ because $D$ is prime, and in this case the Kronecker factorization theorem gives us
\[
L_K(s,\chi_0)\;=\;\zeta(s)L(s,\chi_D). 
\]
Assuming now that $L(s,\chi_D)$ satisfies Hypothesis $\text{H}(c)$, we deduce from the above that each $L_K(s,\chi)$ does as well (assuming that $c\le1/12<0.0875$).
\end{proof}

\begin{remarks}
An explicit value for $c$ as in \eqref{zaman zfr} is not necessary for our result in this article. Without an explicit value of the constant, the above result is due to Fogels \cite{fogels1962distribution}. Additionally, in \cite{zaman2016explicit} they give many other explicit results for more general Hecke $L$-functions. See also \cite{kadiri2012explicit} and \cite{ahn2014some}.

The above proof is precisely the moment where we make use of the fact that $D$ is prime. To work with general fundamental discriminants $-D$, one may adjust the hypothesis of Proposition \ref{prop hyp HC for class group l func} to read ``Suppose that for every divisor $D_1\mid D$, the Dirichlet $L$-function $L(s,\chi_{D_1})$ satisfies Hypothesis $\text{H}(c)$ with $c\le1/12$,'' and the conclusion of the Proposition would still be true. In this case, one would have to correspondingly prove a version of Theorem \ref{exceptional theorem} with a different hypothesis (i.e., ``There exists a constant $c>0$ and a divisor $D_1\mid D$ such that if $L(s,\chi_{D_1})$ has a real zero $\beta$\dots''), which we have chosen not to do here for the sake of a cleaner exposition.
\end{remarks}

\section{The congruence sums}

In this section, we consider the sequence $\mc{A}=(a_n)$ defined in \eqref{definition of a_n} and evaluate the associated congruence sums,
\[
A_d\;=\;\sum_{n\equiv0\smod{d}}a_n\;=\;\sum_{n\equiv0\smod{d}}\frac{1}{n}\lambda_\mc{C}(n)f\Big(\frac{\log n}{\log x}\Big).
\]
Expressing $\lambda_\mc{C}$ in terms of $\lambda_\chi$ by \eqref{orthogonality of characters}, this is accomplished via

\begin{proposition}
\label{corollary summation formula over progressions}
Let $\chi\in\wh{\mc{H}}$ be a class group character. Let $\phi$ be a smooth function supported on $[\alpha_1,\alpha_2]\subset\R^+$, and let $d\le x^{\alpha_1}/D^{3/2}(\log x)^{A+2}$ for some $A>0$. Then
\begin{align*}
\sum_{n\equiv0\smod{d}}\frac{1}{n}\lambda_\chi(n)\phi\Big(\frac{\log n}{\log x}\Big)\;=\;g(d)\cdot(L(1,&\chi_D)\wh{\phi}(0)\log x)\cdot\delta(\chi) \\
&+\;O\Big(\frac{\tau(d)^2}{d}D^{-1/2}(\log x)^{-A}\Big),\nonumber
\end{align*}
where $g(d)$ is the multiplicative function given by
\begin{equation}
\label{density function}
g(d)=\frac{1}{d}\sum_{q\mid d}\frac{\mu(q)}{q}\chi_D(q)\lambda_{\chi_0}\Big(\frac{d}{q}\Big),
\end{equation}
and $\delta(\chi)=1$ if $\chi=\chi_0$, and $\delta(\chi)=0$ otherwise.
\end{proposition}

\begin{remark}
The evaluation of the congruence sums $A_d$ follows directly from the proposition above using \eqref{orthogonality of characters}. In the sieve terminology, this shows that the sequence $\mc{A}=(a_n)$ has level of distribution $y=x^{1-\nu}/D^{3/2}(\log x)^{A+2}$. 
\end{remark}

To prove Proposition \ref{corollary summation formula over progressions}, we use Lemma \ref{summation formula}, a summation formula for the harmonics $\lambda_\chi$ (see for instance (5.16) in \cite{iwaniec2004analytic}). We omit its proof, which is standard---it follows essentially from the functional equation \eqref{functional equation}.

\begin{lemma}
\label{summation formula}
Let $\chi$ be a class group character. Then for a smooth, compactly supported function $\Phi$ on $\R^+$ with Mellin transform $\wt{\Phi}$, we have
\[
\sum_{n\ge1}\lambda_\chi(n)\Phi(n)\;=\;L(1,\chi_D)\wt{\Phi}(1)\delta(\chi)+\frac{1}{\sqrt{D}}\sum_{n\ge1}\lambda_\chi(n)H\Big(\frac{n}{D}\Big),
\]
where $\delta(\chi)=1$ if $\chi=\chi_0$ the trivial character, $\delta(\chi)=0$ if $\chi\ne\chi_0$, and
\begin{equation}
\label{H(y) definition}
H(y)=\frac{1}{2\pi i}\int_{(3)}\wt{\Phi}(1-s)\;\frac{\gamma(s)}{\gamma(1-s)}\;y^{-s}\diff s,
\end{equation}
with $\gamma(s)=(2\pi)^{-s}\Gamma(s)$.
\end{lemma}

Next, we establish a version of the above formula in the logarithmic scale. 

\begin{lemma}
\label{lambda chi summation in log scale}
Let $\phi$ be a smooth function with support contained in $[\alpha_1,\alpha_2]\subset\R^+$. Then for any $k\le x^{\alpha_1}/D$, we have
\[
\sum_{\ell\ge1}\frac{1}{\ell}\lambda_\chi(\ell)\phi\Big(\frac{\log k\ell}{\log x}\Big)\;=\;(L(1,\chi_D)\wh{\phi}(0)\log x)\cdot\delta(\chi)+O(D^{-1/2}(\log x)^{-A})
\]
for any $A>0$, where the implied constant depends only on $\phi$ and $A$. 
\end{lemma}

\begin{proof}
We apply Lemma \ref{summation formula} with the choice
\[
\Phi(\ell)=\frac{1}{\ell}\phi\Big(\frac{\log k\ell}{\log x}\Big).
\]
It is straightforward to verify that
\[
\wt{\Phi}(1)=\wh{\phi}(0)\log x,
\]
and by partial integration $m$ times we derive
\[
\wt{\Phi}(1-s)\;\ll\; (\log x)^{1-m}\Big(\frac{k}{x^{\alpha_1}}\Big)^\sigma(1+|s|)^{-m},
\]
where $\sigma=\Re(s)$, and $m$ is at our disposal. By Stirling's formula, we have
\[
\frac{\gamma(s)}{\gamma(1-s)}\ll t^{2\sigma-1}\qquad\text{for fixed }\sigma.
\]
Now we estimate the function $H(y)$ given by \eqref{H(y) definition}: we move the line of integration to $\Re(s)=\sigma_0$ and take $m>2\sigma_0$ to get
\[
H(y)\ll(\log x)^{1-2\sigma_0}\Big(\frac{k}{yx^{\alpha_1}}\Big)^{\sigma_0}.
\]
Now $|\lambda_\chi(n)|\le\tau(n)$, so as long as $k\le x^{\alpha_1}/D$, we get
\begin{align*}
\frac{1}{\sqrt{D}}\sum_{n\ge1}\lambda_\chi(n)H\Big(\frac{n}{D}\Big)\;&\ll\;D^{-1/2}(\log x)^{1-2\sigma_0}\Big(\frac{kD}{x^{\alpha_1}}\Big)^{\sigma_0}\sum_{n\ge1}\tau(n)n^{-\sigma_0}\\
&\ll\;D^{-1/2}(\log x)^{-A}
\end{align*}
for any given $A>0$ by taking $\sigma_0\ge2$ sufficiently large.
\end{proof}

Note that the main term in the above lemma does not depend on $k$, which is a convenient feature in the forthcoming transformations. Using the lemma above, we prove Proposition \ref{corollary summation formula over progressions}.

\begin{proof}[Proof of Proposition \ref{corollary summation formula over progressions}]
We write $n=dm$ and use the Hecke relations \eqref{Hecke relations}, and then we put $m=q\ell$ to get
\begin{align}
\label{using Hecke relations}
\sum_{n\equiv0\smod{d}}\frac{1}{n}\lambda_\chi(n)&\phi\Big(\frac{\log n}{\log x}\Big) \\
&=\;\frac{1}{d}\sum_{q\mid d}\frac{\mu(q)}{q}\chi_D(q)\lambda_\chi\Big(\frac{d}{q}\Big)\sum_{\ell\ge1}\frac{1}{\ell}\lambda_\chi(\ell)\phi\Big(\frac{\log dq\ell}{\log x}\Big). \nonumber
\end{align}
Next we split the $q$-sum according to whether $q\le Q$ or $q>Q$, with $Q$ to be chosen later. For the latter range where $q>Q$, we estimate the $\ell$-sum trivially using $|\lambda_\chi(\ell)|\le\tau(\ell)$, which gives us
\[
\sum_{\ell\ge1}\frac{1}{\ell}\tau(\ell)\phi\Big(\frac{\log qd\ell}{\log x}\Big)\quad\ll\quad (\log x)^2,
\]
whose contribution to \eqref{using Hecke relations} is
\begin{equation}
\label{range for large q lambda chi}
\frac{1}{d}\sum_{\substack{q\mid d \\ q>Q}}\Big|\frac{\mu(q)}{q}\chi_D(q)\lambda_\chi\Big(\frac{d}{q}\Big)\Big|\sum_{\ell\ge1}\frac{1}{\ell}\tau(\ell)\phi\Big(\frac{\log qd\ell}{\log x}\Big) \quad \ll \quad \frac{\tau(d)^2}{d}\frac{(\log x)^2}{Q}.
\end{equation}
In the other range where $q\le Q$, we apply Proposition \ref{lambda chi summation in log scale}, which is applicable as long as $Qd\le x^{\alpha_1}/D$, which we will arrange for with our later choice of $Q$. This gives
\[
\sum_{\ell\ge1}\frac{1}{\ell}\lambda_\chi(\ell)\phi\Big(\frac{\log qd\ell}{\log x}\Big)\;=\;(L(1,\chi_D)\wh{\phi}(0)\log x)\cdot\delta(\chi)+O(D^{-1/2}(\log x)^{-A-1}).
\]
Plugging the above into \eqref{using Hecke relations} essentially gives the expression for $g(d)$ in \eqref{density function}, except that the $q$-sum here is restricted to $q\le Q$. The range is easily extended to all $q$ up to the same error term in \eqref{range for large q lambda chi}. Thus in total we have now shown
\begin{align*}
\sum_{n\equiv0\smod{d}}\frac{1}{n}\lambda_\chi(n)\phi\Big(\frac{\log n}{\log x}\Big) \; = \; g(d)\cdot &(L(1,\chi_D)\wh{\phi}(0)\log x)\cdot\delta(\chi) \\
&+O\Big(\frac{\tau(d)^2}{d}\Big(\frac{(\log x)^2}{Q}+\frac{\log Q}{D^{1/2}(\log x)^{A+1}}\Big)\Big).
\end{align*}
Finally, choosing $Q=D^{1/2}(\log x)^{A+2}$ gives the result.
\end{proof}

\section{Sums of $\lambda_\chi$ twisted by sieve weights}

In this section, let $[\alpha_1,\alpha_2]\subset\R^+$, and let $\xi^\pm=(\xi^\pm_d)$ be the beta-sieve weights (upper- or lower-bound) of level $y\le x^{\alpha_1}/D^{3/2}(\log x)^{9}$; put $\theta^\pm=1*\xi^\pm=(\theta_n^\pm)$.

\begin{proposition}
\label{lambda chi with sieve weights}
Let $\phi$ be a smooth function supported on $[\alpha_1,\alpha_2]$, and let $\beta=\beta(2)=4.83398\dots$ be the constant from Proposition \ref{fundamental lemma}. Let $z\ge2$ and put $s=\log y/\log z$. Then for $\chi\ne\chi_0$ we have
\[
\sum_{n\ge1}\frac{\theta_n^\pm}{n}\lambda_\chi(n)\phi\Big(\frac{\log n}{\log x}\Big)\;\ll\;h^{-1}(\log x)^{-2}.
\]
For $\chi=\chi_0$ we have
\[
\sum_{n\ge1}\frac{\theta_n^+}{n}\lambda_{\chi_0}(n)\phi\Big(\frac{\log n}{\log x}\Big)\;\le\;XV(z)\Big\{\mf{F}(s)+O((\log y)^{-1/6})\Big\}
\]
for $s\ge\beta-1$, and
\[
\sum_{n\ge1}\frac{\theta_n^-}{n}\lambda_{\chi_0}(n)\phi\Big(\frac{\log n}{\log x}\Big)\;\ge\;XV(z)\Big\{\mf{f}(s)+O((\log y)^{-1/6})\Big\}
\]
for $s\ge\beta$, where
\begin{equation}
\label{X and V(z)}
X\;=\;L(1,\chi_D)\wh{\phi}(0)(\log x),\qquad V(z)\;=\;\prod_{p<z}(1-g(p)),
\end{equation}
and $g(d)$ is given by \eqref{density function}.
\end{proposition}

\begin{proof}
We have
\[
\sum_{n\ge1}\frac{\theta_n^\pm}{n}\lambda_\chi(n)\phi\Big(\frac{\log n}{\log x}\Big)\;=\;\sum_{1\le d\le y}\xi_d\sum_{n\equiv0\smod{d}}\frac{1}{n}\lambda_\chi(n)\phi\Big(\frac{\log n}{\log x}\Big).
\]
Applying Proposition \ref{corollary summation formula over progressions} (with $A=7$) for the $n$-sum on the right-hand side then shows that when $\chi\ne\chi_0$, the above is bounded by (recall that $|\xi_d|\le1$)
\begin{equation}
\label{nontrivial character contribution}
D^{-1/2}(\log x)^{-7}\sum_{1\le d\le y}|\xi_d|\frac{\tau(d)^2}{d} \;\; \ll\;\; D^{-1/2}(\log x)^{-3} \;\;\ll\;\; h^{-1}(\log x)^{-2},
\end{equation}
after using the bound $h=h(D)\ll D^{1/2}\log D$. For $\chi=\chi_0$, we observe that
\[
S^\pm(\mc{A},z)\;=\;\sum_{n\ge1}\frac{\theta_n^\pm}{n}\lambda_{\chi_0}(n)\phi\Big(\frac{\log n}{\log x}\Big)
\]
is the sifted sum for the sequence $\mc{A}=(a_n)$ given by
\[
a_n\;=\;\frac{1}{n}\lambda_{\chi_0}(n)\phi\Big(\frac{\log n}{\log x}\Big).
\]
By Proposition \ref{corollary summation formula over progressions}, the congruence sums for this sequence are
\[
A_d(x)\;=\;\sum_{n\equiv0\smod{d}}\frac{1}{n}\lambda_{\chi_0}(n)\phi\Big(\frac{\log n}{\log x}\Big) \;=\; g(d)X+r_d(x),
\]
where
\[
X\;=\;L(1,\chi_D)\wh{\phi}(0)\log x,\qquad r_d(x)\;\ll\; \frac{\tau(d)^2}{d}D^{-1/2}(\log x)^{-7},
\]
and $g(d)$ is given by \eqref{density function}. We have
\begin{equation}
\label{bound on 1-g}
1-g(p)\;=\;\Big(1-\frac{1}{p}\Big)\Big(1-\frac{\chi_D(p)}{p}\Big)\; \ge\;\Big(1-\frac{1}{p}\Big)^2,
\end{equation}
so we can take $\kappa=2$ in \eqref{sieve dimension condition} for some $K>1$.  Applying Proposition \ref{fundamental lemma} gives
\begin{align*}
S^+(\mc{A},z) &\le XV(z)\Big\{\mf{F}(s)+O((\log y)^{-1/6})\Big\}\;+\;R^+(y,z)\qquad\text{for }s\ge\beta-1, \\
S^-(\mc{A},z) &\ge XV(z)\Big\{\mf{f}(s)+O((\log y)^{-1/6})\Big\}\;+\;\;R^-(y,z)\qquad\text{for }s\ge\beta.
\end{align*}
Just as in \eqref{nontrivial character contribution}, the remainder terms $R^\pm$ are each bounded by
\[
R(y,z)\;=\;\sum_{\substack{d\mid P(z)\\d<y}}|r_d(x)|\;\ll\; D^{-1/2}(\log x)^{-3},
\]
which is covered by $O((\log y)^{-1/6})$, and the proof is complete.
\end{proof}

Here we state a corollary of the above proposition that is ready-to-use for the applications below. From here on we take $D=x^\theta$, $0<\theta<1$, and $z=x^{1/r}$.

\begin{corollary}
\label{corollary lambda chi with sieve weights}
Let $0<\eps<1/10r$, and let $\phi$ be a smooth function supported on $[\alpha_1-\eps,\alpha_2]\subset\R^+$. Let $\xi^\pm$ be the beta-sieve weights of level $y=x^{\alpha_1-\eps}/D^{3/2}$. Suppose that
\[
s\;\coloneqq\;(\alpha_1-\tfrac{3}{2}\theta)r\;\ge\;5.
\]
Then for $x$ sufficiently large we have
\begin{align*}
\sum_{n\ge1}\frac{\theta_n^\pm}{n}\lambda_\chi(n)\phi\Big(\frac{\log n}{\log x}\Big)\;&\ll\;h^{-1}(\log x)^{-2}\qquad\text{if}\quad\chi\ne\chi_0, \\
\sum_{n\ge1}\frac{\theta_n^+}{n}\lambda_{\chi_0}(n)\phi\Big(\frac{\log n}{\log x}\Big)\;&\le\;XV(z)\Big\{1+O(e^{-s})\Big\},\quad\text{and} \\
\sum_{n\ge1}\frac{\theta_n^-}{n}\lambda_{\chi_0}(n)\phi\Big(\frac{\log n}{\log x}\Big)\;&\ge\;XV(z)\Big\{1+O(e^{-s})\Big\},
\end{align*}
where $X$ and $V(z)$ are given by \eqref{X and V(z)}. Furthermore, if Hypothesis $\text{H}(c)$ holds for $L(s,\chi_D)$, then we can replace $XV(z)$ in the above with
\[
XV(z)\;=\;e^{-\gamma}\wh{\phi}(0)r\Big\{1+O(e^{-c/3r\theta})\Big\}.
\]
\end{corollary}

\begin{proof}
This follows directly from Proposition \ref{lambda chi with sieve weights}, with $s=\log y/\log z$ there taken to be $s-\eps r=(\alpha_1-\eps-\frac{3}{2}\theta)r$. The condition $s\ge5$ implies $s-\eps r\ge\beta$ (and $\beta-1$ for the lower bound), since $\eps<1/10r$.

If Hypothesis $\text{H}(c)$ holds for $L(s,\chi_D)$, we may use the estimate \eqref{approximation to L1chiD}; this together with Mertens' theorem shows that
\begin{align*}
XV(z)\;&=\;e^{-\gamma}\wh{\phi}(0)\frac{\log x}{\log z}\Big\{1+O\Big(\exp\Big(\frac{-c\log z}{3\log D}\Big)+\frac{1}{\log z}\Big)\Big\} \\
&=\;e^{-\gamma}\wh{\phi}(0)r\Big\{1+O(e^{-c/3r\theta})\Big\}.
\end{align*}
\end{proof}

\section{A large sieve-type inequality for $\lambda_\chi$ over almost-primes}
\label{sec:ls inequality}

In this section we give a type of large sieve inequality for the harmonics $\lambda_\chi(n)$ where $n$ runs over integers with no small prime factors. General large sieve inequalities for Hecke characters in number fields were given by Huxley \cite{huxley1968large} and Schaal \cite{schaal1970large}; results of the same analytic strength but with explicit constants were later established by Schumer \cite{schumer1986large}. 

However, these results are not sufficient for our particular application here. The above results (after specializing to our case) imply a bound
\begin{equation}
\label{initial ls bound}
\sum_{\chi\in\wh{\mc{H}}}\Big|\sum_{n\sim N}\frac{c_n}{n}\lambda_\chi(n)\Big|^2\;\ll\; (\log N)^2
\end{equation}
for any complex numbers $c_n$ satisfying $|c_n|\le1$, as long as $D\le N$. In our case, the variable $n$ has no small prime factors, say $(n,P(z))=1$. Rather than applying the bound \eqref{initial ls bound} by restriction of the coefficients $c_n$, we apply a sieve that allows us to gain two factors of $\log z$, which is crucial for our application here. Our precise result is

\begin{proposition}
\label{large sieve type inequality}
Put $z=x^{1/r}$ and $D=x^\theta$, and suppose that Hypothesis $\text{H}(c)$ holds for $L(s,\chi_D)$. Then as long as $(\alpha_1-\frac{3}{2}\theta)r\ge5$ and $r\theta\ll1$, we have
\begin{equation}
\label{large sieve bound}
\sum_{\chi\in\wh{\mc{H}}}\Big|\sum_{\substack{x^{\alpha_1}\le n\le x^{\alpha_2} \\ (n,P(z))=1}}\frac{c_n}{n}\lambda_\chi(n)\Big|^2 \;\ll\;(\alpha_2-\alpha_1)^2r^2,
\end{equation}
where $c_n$ are any complex numbers satisfying $|c_n|\le1$. 
\end{proposition}

\begin{proof}
Let $0<\eps<1/10r$. Let $0\le\phi(u)\le1$ be a smooth function supported on $[\alpha_1-\eps,\alpha_2+\eps]$ such that $\phi(u)=1$ for $\alpha_1\le u\le\alpha_2$, and let $\xi^+=(\xi_d^+)$ be the upper-bound beta-sieve weights of level $y=x^{\alpha_1-\eps}/D^{3/2}$.

Expanding the square and rearranging, the left-hand side of \eqref{large sieve bound} is
\[
\mathop{\sum\sum}_{\substack{x^{\alpha_1}\le n_1,n_2\le x^{\alpha_2} \\ (n_1n_2,P(z))=1}}\frac{c_{n_1}\con{c_{n_2}}}{n_1n_2}\sum_{\chi\in\wh{\mc{H}}}\lambda_\chi(n_1)\con{\lambda_\chi}(n_2).
\]
Using the definition \eqref{definition of lambda chi} of $\lambda_\chi$, the above sum over $\chi$ is equal to
\[
\sum_{\chi\in\wh{\mc{H}}}\lambda_\chi(n_1)\con{\lambda_\chi}(n_2)=\sum_{\N\mf{a}_1=n_1}\sum_{\N\mf{a}_2=n_2}\sum_\chi\chi(\mf{a}_1)\con{\chi}(\mf{a}_2)=h\mathop{\sum_{\N\mf{a}_1=n_1}\sum_{\N\mf{a}_2=n_2}}_{\mf{a}_1\sim\mf{a}_2}1,
\]
where $\mf{a}_1\sim\mf{a}_2$ means $\mf{a}_1$ and $\mf{a}_2$ are in the same ideal class. In particular, the above sum is real and nonnegative. Therefore, taking absolute values, we have
\[
\sum_{\chi\in\wh{\mc{H}}}\Big|\sum_{\substack{x^{\alpha_1}\le n\le x^{\alpha_2} \\ (n,P(z))=1}}\frac{c_n}{n}\lambda_\chi(n)\Big|^2\;\le\;\mathop{\sum\sum}_{\substack{x^{\alpha_1}\le n_1,n_2\le x^{\alpha_2} \\ (n_1n_2,P(z))=1}}\frac{1}{n_1n_2}\sum_{\chi\in\wh{\mc{H}}}\lambda_\chi(n_1)\con{\lambda_\chi}(n_2). 
\]
The right-hand side above is majorized by
\[
\sum_{n_1}\sum_{n_2}\frac{\theta_{n_1}^+\theta_{n_2}^+}{n_1n_2}\phi\Big(\frac{\log n_1}{\log x}\Big)\phi\Big(\frac{\log n_1}{\log x}\Big)\sum_{\chi\in\wh{\mc{H}}}\lambda_\chi(n_1)\con{\lambda_\chi}(n_2),
\]
and so we have shown that
\begin{equation}
\label{ls result of lemma}
\sum_{\chi\in\wh{\mc{H}}}\Big|\sum_{\substack{x^{\alpha_1}\le n\le x^{\alpha_2} \\ (n,P(z))=1}}\frac{c_n}{n}\lambda_\chi(n)\Big|^2 \;\le\; \sum_{\chi\in\wh{\mc{H}}}\Big|\sum_{n\ge1}\frac{\theta_n^+}{n}\lambda_\chi(n)\phi\Big(\frac{\log n}{\log x}\Big)\Big|^2.
\end{equation}
Next we apply Corollary \ref{corollary lambda chi with sieve weights} to evaluate the $n$-sum on the right-hand side of the above, which gives
\[
\sum_{n\ge1}\frac{\theta_n^+}{n}\lambda_\chi(n)\phi\Big(\frac{\log n}{\log x}\Big)\;\le\;e^{-\gamma}\wh{\phi}(0)r\Big\{1+O(e^{-c/3r\theta})\Big\}\cdot\delta(\chi)+O(h^{-1}(\log x)^{-2}),
\]
as long as $(\alpha_1-\frac{3}{2}\theta)r\ge5$. Plugging this into \eqref{ls result of lemma}, we have shown
\[
\sum_{\chi\in\wh{\mc{H}}}\Big|\sum_{\substack{x^{\alpha_1}\le n\le x^{\alpha_2} \\ (n,P(z))=1}}\frac{c_n}{n}\lambda_\chi(n)\Big|^2 \;\le\; e^{-2\gamma}\wh{\phi}(0)^2r^2\Big\{1+O(e^{-c/3r\theta})\Big\}.
\]
We have $\wh{\phi}(0)\ll(\alpha_2-\alpha_1)$, and the $O(e^{-c/3r\theta})$ term is superfluous since we assume $r\theta\ll1$. This completes the proof.
\end{proof}

\section{The exceptional case: proof of Theorem \ref{exceptional theorem}}
\label{sec:exceptional}

\begin{proof}[Proof of Theorem \ref{exceptional theorem}]
We begin by applying the Buchstab formula for the sequence $\mc{A}=(a_n)$ given by \eqref{definition of a_n}. This gives
\begin{equation}
\label{applying buchstab formula}
S(\mc{A},\sqrt{x})\;=\;S(\mc{A},z)\;-\sum_{z\le p<\sqrt{x}}S(\mc{A}_p,p),
\end{equation}
where $z=x^{1/r}$, with $r>0$ to be chosen later. Let $\xi^-=(\xi_d^-)$ be the lower-bound beta-sieve weights of level $y=x^{1-\nu}/D^{3/2}$, and put $\theta^-=1*\xi^-=(\theta_n^-)$. The orthogonality \eqref{orthogonality of characters} of the class group characters $\chi\in\wh{\mc{H}}$ implies
\[
S(\mc{A},z)\;\ge\;\frac{1}{h}\sum_{\chi}\con{\chi}(\mc{C})\sum_{n\ge1}\frac{\theta_n^-}{n}\lambda_\chi(n)f\Big(\frac{\log n}{\log x}\Big).
\]
Now we apply Corollary \ref{corollary lambda chi with sieve weights} to get
\[
S(\mc{A},z)\;\ge\;\frac{1}{h}XV(z)\Big\{1+O(e^{-s})\Big\},
\]
where $X$ and $V(z)$ are given by \eqref{X and V(z)}, as long as
\begin{equation}
\label{s condition exceptional}
s\;=\;(1-\tfrac{3}{2}\theta)r\;\ge\;5.
\end{equation}
Now we consider the second term in \eqref{applying buchstab formula}. For each $p$, we use the upper bound
\[
S(\mc{A}_p,p)\;\le\; S(\mc{A}_p,z).
\]
We have
\begin{equation}
\label{cleaning S(Ap,z)}
S(\mc{A}_p,z)\;=\;\sum_{\substack{(n,P(z))=1 \\ (n,p)=1}}a_{pn}\;+\;O\Big(\frac{(\log x)^2}{pz}\Big),
\end{equation}
the error term appearing from the terms $a_{pn}$ with $p\mid n$. Now we let $\xi^+=(\xi_d^+)$ be the upper-bound beta-sieve weights supported on $d\le x^{1/2-\nu}/D^{3/2}$, putting $\theta^+=1*\xi^+=(\theta^+_n)$. Expanding in terms of the characters $\chi$, we have
\[
S(\mc{A}_p,z)\;\le\;\frac{1}{h}\sum_\chi\con{\chi}(\mc{C})\frac{\lambda_{\chi}(p)}{p}\!\!\sum_{\substack{n\ge1 \\(n,p)=1}}\frac{\theta_n^+}{n}\lambda_\chi(n)f\Big(\frac{\log pn}{\log x}\Big).
\]
We remove the condition $(n,p)=1$ up to the same error term in \eqref{cleaning S(Ap,z)}. Note that $f(\log pn/\log x)$ is supported on $n\ge x^{1/2-\nu}$ for any $p$ since $p<\sqrt{x}$. 
So Corollary \ref{corollary lambda chi with sieve weights} gives
\[
S(\mc{A}_p,z)\;\le\;\frac{\lambda_{\chi_0}(p)}{p}\cdot\frac{1}{h}XV(z)\Big\{1+O(e^{-s+r/2})\Big\}\;+\;O\Big(\frac{(\log x)^2}{pz}\Big),
\]
as long as
\begin{equation}
\label{s and r condition exceptional}
s-r/2 \;=\; (\tfrac{1}{2}-\tfrac{3}{2}\theta)r\;\ge\;5.
\end{equation}
The second error term above is subsumed by the first after summing over $p$, and hence we get
\[
\sum_{z\le p<\sqrt{x}}S(\mc{A}_p,p)\;\le\;\delta(z,\sqrt{x})\cdot \frac{1}{h}XV(z)\Big\{1+O(e^{-s+r/2})\Big\},
\]
where we have put
\[
\delta(z,w)\;=\;\sum_{z\le p<w}\frac{\lambda_{\chi_0}(p)}{p}.
\]
One can show (see for instance \S24.2 of \cite{friedlander2010opera}) that
\[
\delta(z,w)\;\le\;2(1-\beta)[\log w+O(z^{-1/4})]\qquad\text{if }w > z\ge D^2,
\]
where $\beta$ is any real zero of $L(s,\chi_D)$. Assuming that $L(s,\chi_D)$ has a real zero $\beta>1-c/\log D$, we get
\[
\delta(z,\sqrt{x})\;\le\;c\theta^{-1}\;+\;O(z^{-1/4}).
\]
Putting everything together, we have shown that
\[
S(\mc{A},\sqrt{x})\;\ge\;\frac{1}{h}XV(z)\Big\{(1-c\theta^{-1})+O(e^{-(1/2-3\theta/2)r})\Big\}.
\]
We choose values for the parameters, each in terms of $r$. We take
\[
\theta\;=\;1/r^2\quad\text{and}\quad c\;=\;1/r^3.
\]
With these choices, we see that both \eqref{s condition exceptional} and \eqref{s and r condition exceptional} hold for $r$ sufficiently large, and that $S(\mc{A},\sqrt{x})\gg h^{-1}XV(z)$. Finally, by \eqref{bound on 1-g} we have
\[
V(z)\;\ge\;\prod_{p<x}\Big(1-\frac{1}{p}\Big)^2\;\gg\;(\log x)^{-2},
\]
and hence
\[
\frac{1}{h}XV(z) \;\gg\; \frac{\wh{f}(0)}{h}\frac{L(1,\chi_D)}{\log x},
\]
which completes the proof.
\end{proof}

\section{The non-exceptional case: proof of Theorem \ref{nonexceptional theorem}}

In this section we prove Theorem \ref{nonexceptional theorem} under the assumption of the three propositions below, whose proofs we give in the final section. Throughout this entire section and the following two, we assume that Hypothesis $\text{H}(c)$ holds for $L(s,\chi_D)$. 

We begin by applying the Buchstab formula
\[
S(\mc{A},\sqrt{x}) \;= S(\mc{A},z)\;-\;\sum_{z\le p<\sqrt{x}}S(\mc{A}_{p},p), 
\]
where $z=x^{1/r}$, with $r$ to be chosen later. (Note that this $r>0$ is independent from the one in the previous section, and may be chosen to have a different numerical value.) We then apply a second iteration of the Buchstab formula to each term $S(\mc{A}_{p},p)$, which gives
\begin{equation}
\label{sieve type identity}
S(\mc{A},\sqrt{x}) \;=\; S_1(\mc{A})\;+\;S_2(\mc{A})\;+\;S_3(\mc{A}),
\end{equation}
where we have put
\[
S_1(\mc{A}) \;=\; S(\mc{A},z),\qquad S_2(\mc{A}) \;=\; -\sum_{z\le p<\sqrt{x}}S(\mc{A}_{p},z),
\]
and
\[
S_3(\mc{A}) \;= \mathop{\sum\sum}_{z\le p_2<p_1<\sqrt{x}}S(\mc{A}_{p_1p_2},p_2).
\]
We evaluate each term $S_i(\mc{A})$ in the following three propositions.

\begin{proposition}
\label{proposition evaluation of S1}
With $z=x^{1/r}$ and $D=x^\theta$, suppose that
\begin{equation}
\label{sifting variable condition for S1}
s\;\coloneqq\;(1-\tfrac{3}{2}\theta)r\;\ge\;5.
\end{equation}
Then we have
\begin{equation}
\label{evaluation of S1}
S_1(\mc{A})\;=\;\frac{e^{-\gamma}\wh{f}(0)r}{h}\Big\{1+O(e^{-c/3r\theta}+e^{-s})\Big\},
\end{equation}
where the implied constant is absolute.
\end{proposition}

\begin{proposition}
\label{proposition evaluation of S2}
With $z=x^{1/r}$ and $D=x^\theta$, suppose that
\begin{equation}
\label{sifting variable condition for S2}
s-r/2\;=\;(\tfrac{1}{2}-\tfrac{3}{2}\theta)r\;\ge\;5.
\end{equation}
Then we have
\begin{equation}
\label{evaluation of S2}
S_2(\mc{A})\;=\;-\frac{e^{-\gamma}\wh{f}(0)r\log(r/2)}{h}\Big\{1+O(e^{-c/3r\theta}+e^{-s+r/2})\Big\},
\end{equation}
where the implied constant is absolute.
\end{proposition} 

\begin{proposition}
\label{final result for S3}
Take $z=x^{1/r}$ and $D=x^\theta$; suppose $r\ge10$ and $\nu\le1/20$. Let $k\ge1$ be such that
\begin{equation}
\label{sifting variable condition for S3}
(\tfrac{1}{2r}-\tfrac{3}{2}\theta)kr\;\ge\;5\qquad\text{and}\qquad kr\theta\;\ll\;1.
\end{equation}
Then we have
\begin{equation}
\label{evaluation of S3}
S_3(\mc{A})\;=\;W\;+\;O\Big(\frac{1}{h}\Big\{r^2\nu^2+kr(\log r)^2e^{-c/18r\theta}+k^2r^6\nu^{-7/2}e^{-5c/18r\theta}\Big\}\Big),
\end{equation}
where $W$ is an explicit quantity defined in Section \ref{Evaluating S3}, and the implied constant is absolute. Importantly, $W$ does not depend on the ideal class $\mc{C}$ in the definition of the sequence $\mc{A}=(a_n)$. 
\end{proposition}

Assuming these propositions for now, we prove Theorem \ref{nonexceptional theorem}.

\begin{proof}[Proof of Theorem \ref{nonexceptional theorem}]
The sequence $\mc{A}=(a_n)$ defined by
\[
a_n\;=\;\frac{1}{n}\lambda_\mc{C}(n)f\Big(\frac{\log n}{\log x}\Big)
\]
depends on the ideal class $\mc{C}$, so for the moment we write $a_n=a_n(\mc{C})$ to emphasize this dependence. We now define the sequence $\mc{B}=(b_n)$ by
\[
b_n\;\coloneqq\;\frac{1}{h}\sum_{\mc{C}\in\mc{H}}a_n(\mc{C})\;=\;\frac{1}{h}\cdot\frac{1}{n}\lambda_{\chi_0}(n)f\Big(\frac{\log n}{\log x}\Big),
\]
the second inequality following from the orthogonality \eqref{orthogonality of characters} of the class group characters. Applying \eqref{sieve type identity} to $\mc{A}$ and $\mc{B}$ and taking the difference, we get
\[
S(\mc{A},\sqrt{x})\;=\;S(\mc{B},\sqrt{x})\;+\;\sum_{1\le i\le3}\Big(S_i(\mc{A})-S_i(\mc{B})\Big).
\]
For each $i$ we have $S_i(\mc{B})=h^{-1}\sum_{\mc{C}\in\mc{H}}S_i(\mc{A})$. Since each of the right-hand sides of \eqref{evaluation of S1}, \eqref{evaluation of S2}, and \eqref{evaluation of S3} does not depend on the class $\mc{C}$, it follows that each of those equations holds with $S_i(\mc{A})$ replaced by $S_i(\mc{B})$. Therefore we get
\begin{align*}
S_1(\mc{A})-S_1(\mc{B})\;&\ll\;\frac{r\nu}{h}\Big\{e^{-c/3r\theta}+e^{-s}\Big\}, \\
S_2(\mc{A})-S_2(\mc{B})\;&\ll\;\frac{(\log r)r\nu}{h}\Big\{e^{-c/3r\theta}+e^{-s+r/2}\Big\},\qquad\text{and} \\
S_3(\mc{A})-S_3(\mc{B})\;&\ll\;\frac{1}{h}\Big\{r^2\nu^2+kr(\log r)^2e^{-c/18r\theta}+k^2r^6\nu^{-7/2}e^{-5c/18r\theta}\Big\}.
\end{align*}
Now we choose our parameters, each in terms of $r$. (NB: the $r$ and $\theta$ here are chosen independently from those in Section \ref{sec:exceptional}, and $c$ here is fixed.) We take
\begin{equation}
\label{parameter choices}
k\;=\;r^{1/2},\quad \theta\;=\;c/r^2,\quad\text{and}\quad \nu\;=\;e^{-r/20}.
\end{equation}
With these choices, one verifies that the conditions \eqref{sifting variable condition for S1}, \eqref{sifting variable condition for S2}, and \eqref{sifting variable condition for S3} are verified for $r$ sufficiently large, and that we have
\[
\sum_{1\le i\le3}\Big(S_i(\mc{A})-S_i(\mc{B})\Big)\;\ll\;\frac{1}{h}r^2e^{-r/18}.
\]
On the other hand, from \eqref{trivial character formula} we have
\[
S(\mc{B},\sqrt{x})\;=\;\frac{1}{h}\sum_{\sqrt{x}\le p<x}\frac{1+\chi_D(p)}{p}f\Big(\frac{\log p}{\log x}\Big).
\]
From the prime number theorem we have
\[
\sum_{p}\frac{1}{p}f\Big(\frac{\log p}{\log x}\Big)\;=\;\wt{f}(0)+O\Big(\frac{1}{\log x}\Big)\;\gg\;\nu,
\]
and from \eqref{smooth summation for chiD} (which assumes Hypothesis $\text{H}(c)$ holds for $L(s,\chi_D)$) we have
\[
\sum_p\frac{\chi_D(p)}{p}f\Big(\frac{\log p}{\log x}\Big)\;\ll\;\nu^{-2}e^{-(1-\nu)c/\theta}.
\]
Therefore for our choices of parameters \eqref{parameter choices}, we get
\[
S(\mc{B},\sqrt{x})\;\gg\; \frac{\nu}{h}\;=\;\frac{1}{h}e^{-r/20}.
\]
From \eqref{sieve type identity} this implies $S(\mc{A},\sqrt{x})\gg \nu/h$ for $r$ sufficiently large, which completes the proof.
\end{proof}

\section{Proofs of Propositions \ref{proposition evaluation of S1} and \ref{proposition evaluation of S2}}

\subsection{Evaluating $S_1(\mc{A})$}

\begin{proof}[Proof of Proposition \ref{proposition evaluation of S1}]
Using the nonnegativity of the terms $a_n$, we have
\[
\sum_{n\ge1}\frac{\theta^-}{n}\lambda_\mc{C}(n)f\Big(\frac{\log n}{\log x}\Big)\;\le\;S_1(\mc{A})\;\le\;\sum_{n\ge1}\frac{\theta^+}{n}\lambda_\mc{C}(n)f\Big(\frac{\log n}{\log x}\Big),
\]
where $\theta^\pm$ are the beta-sieve weights of level $x^{1-\nu}/D^{3/2}$. We expand $\lambda_\mc{C}(n)$ using the orthogonality \eqref{orthogonality of characters} of the class group characters,
\[
\sum_{n\ge1}\frac{\theta^\pm}{n}\lambda_\mc{C}(n)f\Big(\frac{\log n}{\log x}\Big)\;=\;\frac{1}{h}\sum_\chi\con{\chi}(\mc{C})\sum_{n\ge1}\frac{\theta^\pm}{n}\lambda_\chi(n)f\Big(\frac{\log n}{\log x}\Big),
\]
then we apply Corollary \ref{corollary lambda chi with sieve weights} to each $n$-sum on the right-hand side. We get
\begin{align*}
\sum_{n\ge1}\frac{\theta^-}{n}\lambda_\mc{C}(n)f\Big(\frac{\log n}{\log x}\Big)\;&\ge\;\frac{e^{-\gamma}\wh{f}(0)r}{h}\Big\{1+O(e^{-c/3r\theta}+e^{-s})\Big\},\\
\sum_{n\ge1}\frac{\theta^+}{n}\lambda_\mc{C}(n)f\Big(\frac{\log n}{\log x}\Big)\;&\le\;\frac{e^{-\gamma}\wh{f}(0)r}{h}\Big\{1+O(e^{-c/3r\theta}+e^{-s})\Big\},
\end{align*}
where $s=(1-\tfrac{3}{2}\theta)r\ge5$. This gives the result.
\end{proof}

\subsection{Evaluating $S_2(\mc{A})$}

\begin{proof}[Proof of Proposition \ref{proposition evaluation of S2}]
We evaluate the negative of $S_2(\mc{A})$,
\[
-S_2(\mc{A})\;=\sum_{z\le p<\sqrt{x}}S(\mc{A}_p,z).
\]
To evaluate the terms
\[
S(\mc{A}_p,z)\;=\sum_{(n,P(z))=1}\frac{1}{pn}\lambda_\mc{C}(pn)f\Big(\frac{\log pn}{\log x}\Big),
\]
we first attach sieve weights, putting
\[
S^\pm(\mc{A}_p,z)\;=\;\sum_{n\ge1}\frac{\theta_n^\pm}{pn}\lambda_\mc{C}(pn)f\Big(\frac{\log pn}{\log x}\Big),
\]
where $\theta^\pm=1*\xi^\pm$, and $\xi^\pm=(\xi_d^\pm)$ are upper- and lower-bound beta-sieve weights of level $y=x^{1/2-\nu}/D^{3/2}$. Next, using \eqref{orthogonality of characters} (note that $(p,P(z))=1$, since $p\ge z$), we have
\[
S^\pm(\mc{A}_p,z)\;=\;\frac{1}{h}\sum_\chi\con{\chi}(\mc{C})S^\pm(\mc{A}_p,z;\chi),
\]
where we have put
\[
S^\pm(\mc{A}_p,z;\chi)\;=\;\sum_{n\ge1}\frac{\theta_n^\pm}{pn}\lambda_\chi(pn)f\Big(\frac{\log pn}{\log x}\Big).
\]
To use the multiplicativity of $\lambda_\chi$, we first remove the terms where $p\mid n$. Such terms contribute to the above sum at most
\[
\sum_{\substack{n\ge1 \\ p\mid n}}\frac{\tau(n)}{pn}\tau(pn)f\Big(\frac{\log pn}{\log x}\Big)\;\le\;\sum_{m\ge1}\frac{\tau(p)^3}{p^2}\frac{\tau(m)^2}{m}f\Big(\frac{\log p^2m}{\log x}\Big)\;\ll\;\frac{1}{p^2}(\log x)^4,
\]
and hence we get
\[
S^\pm(\mc{A}_p,z;\chi)\;=\;\frac{\lambda_\chi(p)}{p}\sum_{n\ge1}\frac{\theta_n^\pm}{n}\lambda_\chi(n)f\Big(\frac{\log pn}{\log x}\Big)+O\Big(\frac{1}{p^2}(\log x)^4\Big).
\]
Now for the $n$-sum, we apply Corollary \ref{corollary lambda chi with sieve weights} with $\phi(u)=f(u+u_0)$, where $u_0=\log p/\log x$. Note that $\phi$ is supported on $1/2-\nu\le u\le 1/2$, which accounts for the condition \eqref{sifting variable condition for S2}. Following the same lines as in the proof of Proposition \ref{proposition evaluation of S1}, we get
\[
S(\mc{A}_p,z)\;=\;\frac{\lambda_{\chi_0}(p)}{p}\cdot \frac{e^{-\gamma}\wh{f}(0)r}{h}\Big\{1+O(e^{-c/3r\theta}+e^{-s+r/2})\Big\}+O\Big(\frac{1}{p^2}(\log x)^4\Big).
\]
Next, we sum over $p$. The contribution of the second $O$-term above is at most
\[
\frac{1}{z}(\log x)^4\sum_{z\le p<\sqrt{x}}\frac{1}{p}\;\ll\; z^{-1}(\log x)^5,
\]
which is negligible. For the main term, we evaluate
\[
\sum_{z\le p<\sqrt{x}}\frac{\lambda_{\chi_0}(p)}{p}\;=\sum_{z\le p<\sqrt{x}}\frac{1+\chi_D(p)}{p}\;=\;\log(r/2)+O(e^{-c/3r\theta}),
\]
which follows from Mertens' theorem and Corollary \ref{corollary summation for chiD} (using that Hypothesis $\text{H}(c)$ holds for $L(s,\chi_D)$). Putting everything together completes the proof.
\end{proof}

\section{Proof of Proposition \ref{final result for S3}}
\label{Evaluating S3}

In this section we prove Proposition \ref{final result for S3}, where we evaluate the sum
\begin{equation}
\label{type ii sum}
S_3(\mc{A})\;=\mathop{\sum\sum}_{z\le p_2<p_1<\sqrt{x}}S(\mc{A}_{p_1p_2},p_2) \;= \mathop{\sum\sum\sum}_{\substack{z\le p_2<p_1<\sqrt{x} \\ (b,P(p_2))=1}}a_{p_1p_2b}.
\end{equation}

\begin{proof}[Proof of Proposition \ref{final result for S3}]
We define the quantity
\[
W\;\coloneqq\;\frac{1}{2}\Big(W^+(\chi_0)+W^-(\chi_0)\Big),
\]
where $W^\pm(\chi_0)$ are defined in \eqref{W minus and W plus}. From the definitions of $W^\pm(\chi_0)$ in Section \ref{smooth decoupling}, it is apparent that $W$ does not depend on the ideal class in the definition of the sequence $\mc{A}=(a_n)$. Combining the results of Lemmas \ref{S3 lemma1}, \ref{S3 lemma2}, \ref{bounding main term difference}, and \ref{bounding nonprincipal character contributions} below, we see that \eqref{evaluation of S3} holds, subject to the conditions \eqref{sifting variable condition for S3}.
\end{proof}

In the remainder of this section, we state and prove Lemmas \ref{S3 lemma1}, \ref{S3 lemma2}, \ref{bounding main term difference}, and \ref{bounding nonprincipal character contributions}.

\subsection{First arrangements}

We separate from \eqref{type ii sum} the terms where either $b=1$ or $(b,p_1p_2)>1$, putting
\[
S_3(\mc{A})\;=\;V\;+\;V'\;+\;V'',
\]
where
\[
V \;\coloneqq\; \mathop{\sum\sum\sum}_{\substack{z\le p_2<p_1<\sqrt{x} \\ (b,P(p_2))=1,\; b\ne 1 \\ (b,p_1p_2)=1}}a_{p_1p_2b}
\]
gives the main contribution, and we will show that
\[
V' \;\coloneqq\; \mathop{\sum\sum}_{z\le p_2<p_1<\sqrt{x}}a_{p_1p_2} \quad\text{and}\quad V''\;\coloneqq\;\mathop{\sum\sum\sum}_{\substack{z\le p_2<p_1<\sqrt{x} \\ (b,P(p_2))=1 \\ (b,p_1p_2)>1}}a_{p_1p_2b}
\]
give lesser contributions to $S_3(\mc{A})$. First we estimate $V''$: using $|\lambda_\mc{C}(mn)|\le\tau(mn)\le\tau(m)\tau(n)$, we have
\[
V''\;\le\sum_{z\le p_2<\sqrt{x}}\frac{\tau(p_2)}{p_2}\sum_{z\le p_1<\sqrt{x}}\frac{\tau(p_1)}{p_1}\sum_{\substack{z\le b<\sqrt{x} \\ (b,P(p_2))=1 \\ (b,p_1p_2)>1}}\frac{\tau(b)}{b}.
\]
Write $b=p_ib'$ with $i=1$ or $2$. In either case, the $b$-sum above is bounded by
\begin{equation}
\label{V double prime}
\frac{\tau(p_i)}{z}\sum_{b'\le x}\frac{\tau(b')}{b'}\;\ll\; \frac{(\log x)^2}{z},
\end{equation}
and summing over $p_1,p_2$ shows that the same bound holds for $V''$. 

For $V'$, we have
\[
V'\;=\mathop{\sum\sum}_{z\le p_2<p_1<\sqrt{x}}\frac{1}{p_1p_2}\lambda_{\mc{C}}(p_1p_2)f\Big(\frac{\log p_1p_2}{\log x}\Big).
\]
The support of $f$ forces $p_1p_2\ge x^{1-\nu}$, hence $p_1\ge x^{1-\nu}/p_2\ge x^{1/2-\nu}$, and the same for $p_2$. We now open $\lambda_\mc{C}$ using the class group characters, getting
\begin{align*}
V'&\;\le\;\frac{1}{2}\mathop{\sum\sum}_{\substack{x^{1/2-\nu}\le p_i\le x^{1/2} \\ p_1\ne p_2}}\frac{1}{p_1p_2}\lambda_{\mc{C}}(p_1p_2)\;=\;\frac{1}{2h}\sum_\chi\con{\chi}(\mc{C})\!\!\!\!\!\mathop{\sum\sum}_{\substack{x^{1/2-\nu}\le p_i\le x^{1/2} \\ p_1\ne p_2}}\frac{\lambda_\chi(p_1p_2)}{p_1p_2} \\
&\le\;\frac{1}{2h}\sum_\chi\Big|\sum_{x^{1/2-\nu}\le p\le x^{1/2}}\frac{\lambda_\chi(p)}{p}\Big|^2\;+\;O\Big(\frac{1}{x^{1/2-\nu}}\Big).
\end{align*}
Now we apply Proposition \ref{large sieve type inequality}: we choose the coefficients $|c_n|\le1$ appropriately so that the summation
\[
\sum_{\substack{x^{1/2-\nu}\le n\le x^{1/2} \\ (n,P(z))=1}}\frac{c_n}{n}\lambda_\chi(n)
\]
is supported on prime $n$. Then as long as
\[
s-r/2\;=\;(\tfrac{1}{2}-\tfrac{3}{2}\theta)r\;\ge\;5,
\]
the bound \eqref{large sieve bound} gives
\[
\frac{1}{2h}\sum_\chi\Big|\sum_{x^{1/2-\nu}\le p\le x^{1/2}}\frac{\lambda_\chi(p)}{p}\Big|^2 \;\ll\;\frac{\nu^2r^2}{h}.
\]
Thus we have now established

\begin{lemma}
\label{S3 lemma1}
As long as $(\frac{1}{2}-\frac{3}{2}\theta)r\ge5$ and $r\theta\ll1$, we have
\[
S_3(\mc{A})\;=\;V\;+\;O\Big(\frac{\nu^2r^2}{h}\Big).
\]
\end{lemma}

\begin{remarks}
Here we have used our large sieve-type inequality (Proposition \ref{large sieve type inequality}) to not lose the factor $h$ (or any logarithmic factors) after expanding via the class group characters $\chi\in\wh{\mc{H}}$. 

Lemma \ref{S3 lemma1} indicates that the sum $V'$ does in fact contribute to a positive proportion of the sum $S_3(\mc{A})$. However, this proportion is (so to speak) of a lower order of magnitude (proportional to $\nu^2$) than the full sum $S_3(\mc{A})$ (proportional to $\nu$), which is due to the short range of the variables $p_1,p_2$.  
\end{remarks}

\subsection{A smooth decoupling}
\label{smooth decoupling}

For the remaining terms $V$ from $S_3(\mc{A})$, we have $b\ge p_2$, hence
\[
x \;\ge\; p_1p_2b\;\ge\; p_1p_2^2\;>\;p_2^3,
\]
and so $p_2< x^{1/3}$.  We make a smooth partition of the variable $p_2$ into segments that are geometric in the logarithmic scale. First, we partition the range
\[
z\;\le\; p_2\;<\;x^{1/3}
\]
into segments $z_{j-1}\le p_2\le z_j$ with $j=1,2,\dots, J$, where $z_j$ are given by
\[
z_j = z^{\alpha^j}=x^{\alpha^j/r}, \qquad z^{\alpha^J}=x^{1/3},
\]
so $J\ge1$ is at our disposal, and it determines $\alpha>1$. Now we make a smooth partition from these points $z_j$ in the following way:
\begin{itemize}
\item we now let the index $j$ run over half-integers $\frac{1}{2}, 1, \frac{3}{2}, \dots, J, J+\frac{1}{2}$;
\item for each such $j$, let $0\le h_j(t)\le 1$ be a smooth bump function supported on $[\alpha^{j-1}/r,\alpha^j/r]$, such that for $j=\frac{1}{2},1,\frac{3}{2},\dots,J$ we have
\[
h_{j}(t)+h_{j+1/2}(t)=1\qquad\text{for}\qquad t\in [\alpha^j/r,\alpha^{j+1/2}/r].
\]
\end{itemize}
We put
\[
h^-(t)\;=\sum_{1\le j\le J}h_j(t)\qquad\text{and}\qquad h^+(t)\;=\sum_{\frac{1}{2}\le j\le J+\frac{1}{2}}h_j(t),
\]
and from the above properties we get
\begin{align}
\label{h plus and minus}
h^-(t)\;\le\;\vec{1}_{[1/r,\;1/3]}(t)\;\le\; h^+(t)\qquad&\text{for all }t, \\
h^+(t)\;=\;1\qquad&\text{for }\;\;1/r\le t\le1/3, \nonumber \\
h^-(t)\;=\;1\qquad&\text{for }\;\;\alpha^{1/2}/r\le t\le\alpha^{J+1/2}/r. \nonumber
\end{align}
We now use this smooth partition of unity to decouple the variables $p_1,p_2,b$. From \eqref{h plus and minus} we have
\begin{align*}
V&\;\ge \sum_{1\le j\le J}\sum_{p_2}h_j\Big(\frac{\log p_2}{\log x}\Big) \sum_{p_2<p_1<\sqrt{x}}\sum_{\substack{p_2\le b \\ (b,P(p_2))=1 \\ (b,p_1p_2)=1}}a_{p_1p_2b},\qquad\text{and} \\
V&\;\le \sum_{\frac{1}{2}\le j\le J+\frac{1}{2}}\sum_{p_2}h_j\Big(\frac{\log p_2}{\log x}\Big) \sum_{p_2<p_1<\sqrt{x}}\sum_{\substack{p_2\le b \\ (b,P(p_2))=1 \\ (b,p_1p_2)=1}}a_{p_1p_2b}.
\end{align*}
The conditions
\begin{equation}
\label{p1 and b entangled}
p_1>p_2, \qquad b\ge p_2, \qquad \text{and} \qquad (b,P(p_2))=1
\end{equation}
entangle $p_1$ and $b$ with $p_2$, so we adjust them to decouple these variables. The variable $p_2$ lies in the restricted range $z_{j-1}\le p_2\le z_j$ by the support of $h_j$, and so (by positivity) we replace the three conditions \eqref{p1 and b entangled} respectively by
\[
p_1>z_j, \qquad b\ge z_j, \qquad\text{and}\qquad (b,P(z_j))=1
\]
in the lower bound for $V$, and with
\[
p_1>z_{j-1}, \qquad b\ge z_{j-1},\qquad\text{and}\qquad (b,P(z_{j-1}))=1
\]
in the upper bound for $V$. After these adjustments we have
\begin{align}
\label{V lower bound}
V &\;\ge\; W^-\;\coloneqq \sum_{1\le j\le J}\sum_{p_2}h_j\Big(\frac{\log p_2}{\log x}\Big) \sum_{z_j<p_1<\sqrt{x}}\sum_{\substack{z_j\le b \\ (b,P(z_j))=1 \\ (b,p_1p_2)=1}}a_{p_1p_2b},\qquad\text{and} \\
\label{V upper bound}
V &\;\le\; W^+\;\coloneqq \sum_{\frac{1}{2}\le j\le J+\frac{1}{2}}\sum_{p_2}h_j\Big(\frac{\log p_2}{\log x}\Big) \sum_{z_{j-1}<p_1<\sqrt{x}}\sum_{\substack{z_{j-1}\le b \\ (b,P(z_{j-1}))=1 \\ (b,p_1p_2)=1}}a_{p_1p_2b}.
\end{align}
Now we open $a_{p_1p_2b}$ using characters: by \eqref{definition of a_n} and \eqref{orthogonality of characters} we have
\[
a_n=\frac{1}{h}\sum_{\chi\in\wh{\mc{H}}}\con{\chi}(\mc{C})\frac{\lambda_\chi(n)}{n}f\Big(\frac{\log n}{\log x}\Big),
\]
which we put into \eqref{V lower bound} and \eqref{V upper bound}. By our adjustments above, we have arranged that $p_1,p_2,b$ are automatically pairwise coprime, so the multiplicativity of $\lambda_\chi$ now gives
\begin{align*}
W^- &\;=\; \frac{1}{h}\sum_{1\le j\le J}\sum_\chi\con{\chi}(\mc{C})\sum_{z_j<p_1<\sqrt{x}}\frac{\lambda_{\chi}(p_1)}{p_1} \\
&\phantom{\frac{1}{h}\sum_{1\le j\le J}\sum_\chi\con{\chi}(\mc{C})}\cdot\sum_{\substack{z_j\le b \\ (b,P(z_j))=1 \\ (b,p_1p_2)=1}}\frac{\lambda_{\chi}(b)}{b}\sum_{p_2}\frac{\lambda_{\chi}(p_2)}{p_2}h_j\Big(\frac{\log p_2}{\log x}\Big)f\Big(\frac{\log p_1p_2b}{\log x}\Big),
\end{align*}
and
\begin{align*}
W^+ &\;=\; \frac{1}{h}\sum_{\frac{1}{2}\le j\le J+\frac{1}{2}}\sum_\chi\con{\chi}(\mc{C})\sum_{z_{j-1}<p_1<\sqrt{x}}\frac{\lambda_{\chi}(p_1)}{p_1} \\
&\phantom{\frac{1}{h}\sum_{1\le j\le J}\sum_\chi\con{\chi}(\mc{C})}\cdot\sum_{\substack{z_{j-1}\le b \\ (b,P(z_{j-1}))=1 \\ (b,p_1p_2)=1}}\frac{\lambda_{\chi}(b)}{b}\sum_{p_2}\frac{\lambda_{\chi}(p_2)}{p_2}h_j\Big(\frac{\log p_2}{\log x}\Big)f\Big(\frac{\log p_1p_2b}{\log x}\Big).
\end{align*}
Finally, we wish to remove the conditions $(b,p_1p_2)=1$ now that we have made use of them to decouple $\lambda_\chi(p_1p_2b)$. To do so, we add back the missing terms, which are bounded by the same error term in \eqref{V double prime} from before. Putting
\begin{align*}
W^-_j(\chi) &\;=\sum_{z_j<p_1<\sqrt{x}}\!\!\!\!\!\frac{\lambda_{\chi}(p_1)}{p_1}\!\!\!\!\!\!\sum_{\substack{z_j\le b \\ (b,P(z_j))=1}}\!\!\!\!\!\!\!\frac{\lambda_{\chi}(b)}{b}\sum_{p_2}\frac{\lambda_{\chi}(p_2)}{p_2}h_j\Big(\frac{\log p_2}{\log x}\Big)f\Big(\frac{\log p_1p_2b}{\log x}\Big),\quad\text{and} \\
W^+_j(\chi) &\;=\sum_{z_{j-1}<p_1<\sqrt{x}}\!\!\!\!\!\frac{\lambda_{\chi}(p_1)}{p_1}\!\!\!\!\!\!\!\!\sum_{\substack{z_{j-1}\le b \\ (b,P(z_{j-1}))=1}}\!\!\!\!\!\!\!\!\!\frac{\lambda_{\chi}(b)}{b}\sum_{p_2}\frac{\lambda_{\chi}(p_2)}{p_2}h_j\Big(\frac{\log p_2}{\log x}\Big)f\Big(\frac{\log p_1p_2b}{\log x}\Big),
\end{align*}
we have shown the following

\begin{lemma}
\label{S3 lemma2}
We have
\begin{align}
\label{lemma V into characters lower}
V\;&\gg\;\frac{1}{h}\sum_{1\le j\le J}\sum_\chi\con{\chi}(\mc{C}) W_j^-(\chi),\quad\text{and} \\ 
\label{lemma V into characters upper}
V\;&\ll\;\frac{1}{h}\sum_{\frac{1}{2}\le j\le J+\frac{1}{2}}\sum_\chi\con{\chi}(\mc{C}) W_j^+(\chi).
\end{align}
\end{lemma}

\subsection{The contribution of the principal character}

Now we extract from \eqref{lemma V into characters lower} and \eqref{lemma V into characters upper} the contribution from the principal character $\chi=\chi_0$, which constitutes the main part of $V$. Accordingly we put
\begin{equation}
\label{W minus and W plus}
W^-(\chi_0)=\frac{1}{h}\sum_{1\le j\le J}W^-_j(\chi_0) \qquad\text{and}\qquad W^+(\chi_0)=\frac{1}{h}\sum_{\frac{1}{2}\le j\le J+\frac{1}{2}}W^+_j(\chi_0),
\end{equation}
and we will show that the difference
\[
W^+(\chi_0)-W^-(\chi_0)\;=\;\frac{1}{h}\Big(W_{1/2}^+(\chi_0)+W_{J+1/2}^+(\chi_0)+\sum_{1\le j\le J}(W_j^+(\chi_0)-W_j^-(\chi_0))\Big)
\]
is comparably small.

\begin{lemma}
\label{bounding main term difference}
Let $k\ge1$, and suppose that $\nu\le1/20$, $r\ge10$, and
\begin{equation}
\label{main term difference sifting variable condition}
(\tfrac{1}{2r}-\tfrac{3}{2}\theta)kr\;\ge\;5.
\end{equation}
Then we have
\[
W^+(\chi_0)-W^-(\chi_0)\;\ll\;\frac{1}{h}kr(\log r)^2e^{-c/18r\theta}.
\]
\end{lemma}

\begin{remark}
The variable $k\ge1$ plays no essential theoretical role, but it must be present for technical reasons. On a mechanical level, it is a parameter that can be taken larger to ensure that Corollary \ref{corollary lambda chi with sieve weights} is applicable even when the range of the involved summation includes (relatively) very small integers.
\end{remark}

To prove this lemma, we will use the following couple of estimates.

\begin{lemma}
\label{upper bound for sum over primes}
Let $0<\alpha_1<\alpha_2<1$, and suppose that $L(s,\chi_D)$ satisfies Hypothesis $\text{H}(c)$. Then we have
\[
\sum_{x^{\alpha_1}\le p\le x^{\alpha_2}}\frac{\lambda_{\chi_0}(p)}{p}\;\ll\;\log(\alpha_2/\alpha_1)\;+\;e^{-c\alpha_1/3\theta}.
\]
\end{lemma}

\begin{proof}
We have $\lambda_{\chi_0}(p)=1+\chi_D(p)$, so the result follows directly from Mertens' theorem and Corollary \ref{corollary summation for chiD}.
\end{proof}

\begin{lemma}
\label{upper bound for sum over almost primes}
Let $0<\alpha_1<\alpha_2<1$. Let $k\ge1$, and suppose that
\begin{equation}
\label{condition for sifting variable for b sum}
(\alpha_1-\tfrac{3}{2}\theta)kr\;\ge\;5.
\end{equation}
Then for any $w\ge z^{1/k}=x^{1/kr}$, we have
\begin{equation}
\label{bound for lambda chi0 b over almost primes}
\sum_{\substack{x^{\alpha_1}\le b\le x^{\alpha_2} \\ (b,P(w))=1}}\frac{\lambda_{\chi_0}(b)}{b} \;\ll\; kr.
\end{equation}
\end{lemma}

\begin{proof}
Since $w\ge z^{1/k}$, by positivity the left hand side of \eqref{bound for lambda chi0 b over almost primes} is bounded by
\[
\sum_{\substack{x^{\alpha_1}\le b\le x^{\alpha_2} \\ (b,P(z^{1/k}))=1}}\frac{\lambda_{\chi_0}(b)}{b}\phi\Big(\frac{\log b}{\log x}\Big),
\]
where $0\le\phi(u)\le1$ is a smooth function supported on $[\alpha_1-\eps,\alpha_2+\eps]$ with $0<\eps<1/10kr$, and $\phi(u)=1$ when $\alpha_1\le u\le \alpha_2$. We apply Corollary \ref{corollary lambda chi with sieve weights}.
\end{proof}

\begin{proof}[Proof of Lemma \ref{bounding main term difference}]
In the following, we will apply Lemma \ref{upper bound for sum over almost primes} in situations where the variable $b$ always satisfies $b\ge z_{-1/2}$, i.e., $\alpha_1\ge1/\alpha^{1/2}r\ge1/2r$. Therefore we assume throughout that $k\ge1$ is chosen so that the condition \eqref{main term difference sifting variable condition} holds. This ensures that \eqref{condition for sifting variable for b sum} holds any time we apply Lemma \ref{upper bound for sum over almost primes}. 

First we handle $W^+_{1/2}(\chi_0)$. For these terms we have $z_{-1/2}\le p_2\le z_{1/2}$, or
\[
\frac{1}{r}\alpha^{-1/2}\le\frac{\log p_2}{\log x}\le\frac{1}{r}\alpha^{1/2},
\]
as well as $z_{-1/2}\le p_1\le\sqrt{x}$, or
\begin{equation}
\label{range of p1 in W+1/2}
\frac{1}{r}\alpha^{-1/2}\le\frac{\log p_1}{\log x}\le\frac{1}{2}.
\end{equation}
Using $x^{1-\nu}\le p_1p_2b\le x$ and the above bounds, we get
\begin{equation}
\label{range of b in W+1/2}
\frac{1}{2}-\nu-\frac{1}{r}\alpha^{1/2}\le\frac{\log b}{\log x}\le1-\frac{2}{r}\alpha^{-1/2}.
\end{equation}
Assuming that $\alpha\le2$, $r\ge10$, and $\nu\le1/20$, we replace (by positivity) the inequalities \eqref{range of p1 in W+1/2} and \eqref{range of b in W+1/2} by the simpler ones
\[
\frac{1}{2r}\le\frac{\log p_1}{\log x}\le\frac{1}{2}\qquad\text{and}\qquad\frac{1}{4}\le\frac{\log b}{\log x}\le1.
\]
Now applying Lemmas \ref{upper bound for sum over primes} and \ref{upper bound for sum over almost primes}, we get (using $\alpha^{1/2}<2$)
\begin{equation}
\label{initial W12 bound}
W^+_{1/2}(\chi_0)\;\ll\;kr(\log r)\Big(\log\alpha+e^{-c/6r\theta}\Big).
\end{equation}
Next we analyze $W^+_{J+1/2}(\chi_0)$. We have $\alpha^J/r=1/3$, so for these terms we have
\[
\frac{1}{3}\alpha^{-1/2}\le\frac{\log p_2}{\log x}\le\frac{1}{3}\alpha^{1/2}.
\]
The same lower bounds in \eqref{range of p1 in W+1/2} and \eqref{range of b in W+1/2} hold, and hence $p_1p_2b\le x$ gives
\[
\frac{1}{3}\alpha^{-1/2}\quad\le\quad\frac{\log p_1}{\log x}\;,\;\frac{\log b}{\log x}\quad\le\quad1-\frac{2}{3}\alpha^{-1/2}.
\]
We apply Lemmas \ref{upper bound for sum over primes} and \ref{upper bound for sum over almost primes} again, and we find that $W_{J+1/2}^+(\chi_0)$ satisfies the same bound as $W_{1/2}^+(\chi_0)$, 
\begin{equation}
\label{initial WJ bound}
W_{J+1/2}^+(\chi_0)\;\ll\;kr(\log r)\Big(\log\alpha+e^{-c/6r\theta}\Big).
\end{equation}
The differences $W^+_j(\chi_0)-W^-_j(\chi_0)$ are more complicated, but using positivity we can majorize them by two simpler sums,
\[
W^+_j(\chi_0)-W^-_j(\chi_0)\;\le\; U_1+U_2,
\]
where
\[
U_1 \;=\; \sum_{z_{j-1}\le p_1<z_j}\frac{\lambda_{\chi_0}(p_1)}{p_1}\sum_{z_{j-1}\le p_2< z_j}\frac{\lambda_{\chi_0}(p_2)}{p_2}\sum_{\substack{z_{j-1}\le b \le x \\ (b,P(z_{j-1}))=1}}\frac{\lambda_{\chi_0}(b)}{b}f\Big(\frac{\log p_1p_2b}{\log x}\Big)
\]
and
\[
U_2 \;=\; \sum_{z_{j-1}\le p_1<\sqrt{x}}\frac{\lambda_{\chi_0}(p_1)}{p_1}\sum_{z_{j-1}\le p_2< z_j}\frac{\lambda_{\chi_0}(p_2)}{p_2}\sum_{b\in B_j}\frac{\lambda_{\chi_0}(b)}{b}f\Big(\frac{\log p_1p_2b}{\log x}\Big),
\]
and the set $B_j$ is given by the difference of sets
\[
B_j \;=\; \{b\ge z_{j-1}\;;\;(b,P(z_{j-1}))=1\}\;\setminus\;\{b\ge z_{j}\;;\;(b,P(z_{j}))=1\}.
\]
For $U_1$, the condition $p_1p_2b\le x$ implies
\[
\frac{\log b}{\log x}\;\le\;1-\frac{2}{r}\alpha^{j-1}\;\le\;1-\frac{2}{r}\qquad\text{for }1\le j\le J.
\]
Applying Lemmas \ref{upper bound for sum over primes} and \ref{upper bound for sum over almost primes} then gives
\begin{equation}
\label{U1 simplified}
U_1\;\ll\;kr\Big(\log\alpha+e^{-c/3r\theta}\Big)^2.
\end{equation}
For $U_2$, we observe that if $b\in B_j$ and $b<z_j$, then $b$ must be prime, since $z_{j-1}>z_j^{1/2}$ (assuming $\alpha<2$). Otherwise, the elements of $B_j$ are $b=p_3b'$, where $z_{j-1}\le p_3<z_j$, $b'\ge z_{j-1}$, and $(b',P(z_{j-1}))=1$. In other words, 
\begin{align*}
B_j\;\subseteq\;\Big\{b=p_3b'\;;\;z_{j-1}\le p_3<z_j,\;\; &\text{and either }b'=1 \\
&\text{ or }b'\ge z_{j-1}\text{ and }(b',P(z_{j-1}))=1\Big\}.
\end{align*}
Therefore, using $\lambda_{\chi_0}(p_3b')\le\lambda_{\chi_0}(p_3)\lambda_{\chi_0}(b')$, we get
\[
\sum_{b\in B_j}\frac{\lambda_{\chi_0}(b)}{b}\quad\le\quad\Big(\sum_{z_{j-1}\le p_3<z_j}\frac{\lambda_{\chi_0}(p_3)}{p_3}\Big)\Big(1\;\;+\!\!\!\!\sum_{\substack{b'\ge z_{j-1} \\ (b',P(z_{j-1}))=1}}\frac{\lambda_{\chi_0}(b')}{b'}\Big).
\]
Using the condition $p_1p_2p_3b'\le x$, we apply Lemmas \ref{upper bound for sum over primes} and \ref{upper bound for sum over almost primes} to get
\[
U_2\;\ll\;kr(\log r)\Big(\log\alpha+e^{-c/3r\theta}\Big)^2,
\]
which we combine with \eqref{U1 simplified} and sum over $j$ to get
\[
\sum_{1\le j\le J}(W_j^+(\chi_0)-W_j^-(\chi_0))\;\ll\;Jkr(\log r)\Big(\log\alpha+e^{-c/3r\theta}\Big)^2.
\]
Now we make a choice for the parameter $J$: we take
\begin{equation}
\label{choice of J}
J\;=\;\Big[(\log r)e^{c/18r\theta}\Big]+1,
\end{equation}
where $[\;\cdot\;]$ denotes the integer part. This determines $\alpha$ via the relation
\[
\alpha^J\;=\;r/3,\quad\text{or}\quad \log\alpha=J^{-1}\log(r/3),
\]
and hence our choice of $J$ implies
\begin{equation}
\label{bound for log alpha}
\log\alpha\;\ll\;e^{-c/18r\theta}\qquad\text{and}\qquad J\Big(\log\alpha+e^{-c/3r\theta}\Big)\;\ll\;\log r.
\end{equation}
Therefore we have
\[
\sum_{1\le j\le J}(W_j^+(\chi_0)-W_j^-(\chi_0))\;\ll\;kr(\log r)^2e^{-c/18r\theta}.
\]
Similarly, from \eqref{initial W12 bound} and \eqref{initial WJ bound} and the bound \eqref{bound for log alpha}, we see that both $W_{1/2}^+(\chi_0)$ and $W_{J+1/2}^+(\chi_0)$ satisfy the same bound. This gives the result. 
\end{proof}

\subsection{The contribution of the other characters}
\label{Evaluating V}

In this section we estimate the contributions of the nonprincipal characters $\chi\ne\chi_0$ to the lower and upper bounds \eqref{lemma V into characters lower} and \eqref{lemma V into characters upper}.

\begin{lemma}
\label{bounding nonprincipal character contributions}
Let $k\ge1$ be such that \eqref{main term difference sifting variable condition} holds and $kr\theta\ll1$. Then
\begin{align*}
\frac{1}{h}\sum_{j}\sum_{\chi\ne\chi_0}\con{\chi}(\mc{C})W^\pm_j(\chi)\;\ll\;\frac{1}{h}k^2r^6\nu^{-7/2}e^{-5c/18r\theta},
\end{align*}
where the $j$-sum runs over $1\le j\le J$ for $W^-$ and $\frac{1}{2}\le j\le J+\frac{1}{2}$ for $W^+$. 
\end{lemma}

\begin{proof}
We treat both $W^\pm$ at the same time because our arguments do not depend on the specific ranges of the variables $p_1$ and $b$, only that the inequalities
\[
p_1,b\;\ge\; z_{-1/2}
\]
hold for every $j$. First, we evaluate the sum over $p_2$ using the explicit formula. Specifically, we apply Proposition \ref{proposition explicit formula in log scale} with the choice
\[
\phi(u)\;=\;h_j(u)f\Big(u+\frac{\log p_1}{\log x}+\frac{\log b}{\log x}\Big),
\]
which gives us
\[
W^\pm_j(\chi)\;=\sum_{p_1}\frac{\lambda_{\chi}(p_1)}{p_1}\sum_{b}\frac{\lambda_{\chi}(b)}{b}\;\sum_{\rho_\chi}\wt{\Phi}(\rho_\chi)\;+\;O\Big(\frac{1}{h\log x}\Big),
\]
where
\[
\wt{\Phi}(s)\;=\;\int_{-\infty}^\infty x^{u(s-1)}u^{-1}h_j(u)f\Big(u+\frac{\log p_1}{\log x}+\frac{\log b}{\log x}\Big)\diff u
\]
and the sum $\sum_{\rho_\chi}$ runs over the nontrivial zeros of $L_K(s,\chi)$. (We don't write the conditions for $p_1,b$ explicitly, since they differ for $W^\pm$, but of course we keep them in our minds as necessary.) Note that there is no polar contribution as in \eqref{explicit formula in log scale} since here we treat the nonprincipal characters $\chi\ne\chi_0$. We decouple the variables $p_1,b$ from the above integral by applying the Fourier inversion
\[
f(u+\delta)=\int_{-\infty}^\infty\wh{f}(w)\e((u+\delta)w)\diff w,
\]
which gives us
\begin{align*}
W^\pm_j(\chi)\;=\int_{-\infty}^\infty\wh{f}(w)\sum_{p_1}&\e\Big(w\frac{\log p_1}{\log x}\Big)\frac{\lambda_{\chi}(p_1)}{p_1} \\
&\sum_{b}\e\Big(w\frac{\log b}{\log x}\Big)\frac{\lambda_{\chi}(b)}{b}\sum_{\rho_\chi}H(\rho_\chi,w)\diff w\;+\;O\Big(\frac{1}{h\log x}\Big),
\end{align*}
where
\[
H(s,w)\;=\;\int_{-\infty}^\infty x^{u(s-1)}u^{-1}h_j(u)\e(wu)\diff u.
\]
Summing over $\chi\ne\chi_0$ and taking the absolute value, we get
\[
\Big|\sum_{\chi\ne\chi_0}\con{\chi}(\mc{C})W^\pm_j(\chi)\Big|\;\le\;\int_{-\infty}^\infty|\wh{f}(w)| H(w) K(w)\diff w,
\]
where we have put
\begin{equation}
\label{H(w) sum over zeros}
H(w)\;=\;\max_{\chi\ne\chi_0}\Big|\sum_{\rho_\chi}H(\rho_\chi,w)\Big|
\end{equation}
and
\begin{equation}
\label{K(w) definition}
K(w)\;=\;\sum_{\chi\ne\chi_0}\Big|\sum_{p_1}\e\Big(w\frac{\log p_1}{\log x}\Big)\frac{\lambda_{\chi}(p_1)}{p_1}\Big|\Big|\sum_{b}\e\Big(w\frac{\log b}{\log x}\Big)\frac{\lambda_{\chi}(b)}{b}\Big|.
\end{equation}
To estimate $H(w)$, we integrate by parts three times (after borrowing/returning a factor $e^u$) to get
\[
H(s,w)\;=\;(1+(1-s)\log x)^{-3}\int_{-\infty}^\infty e^{-u}x^{u(s-1)}(e^uu^{-1}h_j(u)\e(wu))'''\diff u.
\]
For $k=1,2,3$ we have
\[
h_j^{(k)}(u) \;\ll\; \Big(\frac{\alpha^j}{r}-\frac{\alpha^{j-1}}{r}\Big)^{-k}\;\ll\; (r/\alpha^j\log\alpha)^k\;\le\;(r/\log\alpha)^k
\]
for every $j$, with $\alpha^{j-1}/r\le u\le\alpha^j/r$, and thus we get
\[
(e^uu^{-1}h_j(u)\e(wu))'''\;\ll\;e^u(r/\log\alpha)^3(1+|w|^3),
\]
which holds for all $w$. Using this to estimate the above integral, we derive
\[
H(s,w)\;\ll\;|1+(1-s)\log x|^{-3}x^{(\sigma-1)/2r}(r/\log\alpha)^3(1+|w|^3).
\]
We use this bound to estimate \eqref{H(w) sum over zeros}. The number of zeros $\rho_\chi=\beta_\chi+i\gamma_\chi$ of $L_K(s,\chi)$ with $0<\beta_\chi<1$ and $t<|\gamma_\chi|\le t+1$ is $O(\log D(|t|+1))$, so the above bound implies
\begin{align*}
\sum_{k\ge1}\sum_{k<|\gamma_\chi|\le k+1}|H(\rho_\chi,w)|\;&\ll\;(r/\log\alpha)^3(1+|w|^3)\sum_{k\ge1}(\log Dk)k^{-3}(\log x)^{-3} \\
&\ll\;(r/\log\alpha)^3(1+|w|^3)(\log x)^{-2}.
\end{align*}
By Proposition \ref{prop hyp HC for class group l func}, the remaining zeros $\rho_\chi$ of $L_K(s,\chi)$ fall in the range of Hypothesis $\text{H}(c)$, so they satisfy
\[
\beta_\chi\;\le\;1-\frac{c}{\log D},
\]
and hence $x^{(1-\beta_\chi)/2r}\le e^{-c/2r\theta}$. Applying Lemma \ref{zero density surrogate estimate} now gives
\begin{equation}
\label{bound for H(rho,w) on low zeros}
\sum_{\substack{\rho_\chi=\beta_\chi+i\gamma_\chi \\ |\gamma_\chi|\le1}}|H(\rho_\chi,w)|\;\ll\;(r/\log\alpha)^3e^{-c/2r\theta}(1+|w|^3).
\end{equation}
This bound covers the one above for $|\gamma_\chi|>1$, and it is uniform in $\chi$, so we deduce that $H(w)$ is also bounded by the right-hand side of \eqref{bound for H(rho,w) on low zeros}.

For $K(w)$, we apply the Cauchy inequality to the $\chi$-sum in \eqref{K(w) definition} to get
\begin{align}
\label{K(w) after Cauchy}
K(w)\;\le\;\Big(\sum_\chi\Big|\sum_{p_1}&\e\Big(w\frac{\log p_1}{\log x}\Big)\frac{\lambda_{\chi}(p_1)}{p_1}\Big|^2\Big)^{1/2} \\
&\cdot\Big(\sum_\chi\Big|\sum_{b}\e\Big(w\frac{\log b}{\log x}\Big)\frac{\lambda_{\chi}(b)}{b}\Big|^2\Big)^{1/2}. \nonumber
\end{align}
Here we have (by positivity) added back the principal character to the $\chi$-sums. We apply Proposition \ref{large sieve type inequality}: with $k\ge1$ chosen so that \eqref{main term difference sifting variable condition} holds, we have
\[
K(w)\;\le\;\Big(\sum_\chi\Big|\!\!\!\!\!\!\sum_{\substack{m\ge1 \\ (m,P(x^{1/kr}))=1}}\!\!\!\!\!\!c_m\frac{\lambda_{\chi}(m)}{m}\Big|^2\Big)^{1/2}\Big(\sum_\chi\Big|\!\!\!\!\!\!\sum_{\substack{n\ge1 \\ (n,P(x^{1/kr}))=1}}\!\!\!\!\!\!c_n'\frac{\lambda_{\chi}(n)}{n}\Big|^2\Big)^{1/2},
\]
where we choose the coefficients $c_m,c_n'$ to agree with those of $p_1,b$ in \eqref{K(w) after Cauchy} when $m=p_1,n=b$, and we choose them to be 0 otherwise. This gives
\[
K(w)\;\ll\;k^2r^2.
\]
For the integration over $w$, we have
\begin{align*}
\int_{-\infty}^\infty|\wh{f}(w)|(1&+|w|^3)\diff w \\
&\le\; \Big(\int_{-\infty}^\infty|\wh{f}(w)|^2(1+|w|^4)^2\diff w\Big)^{1/2}\Big(\int_{-\infty}^\infty\Big(\frac{1+|w|^3}{1+|w|^4}\Big)^2\diff w\Big)^{1/2} \\
&\ll\;\Big(\int_{-\infty}^\infty(|f(u)|^2+|f^{(2)}(u)|^2+|f^{(4)}(u)|^2)\diff u\Big)^{1/2}.
\end{align*}
The derivatives $f^{(k)}(u)$ are supported on $1-\nu\le u\le1$ and satisfy
\[
\max_u|f^{(k)}(u)|\;\ll\;\nu^{-k},
\]
and hence we derive
\[
\int_{-\infty}^\infty|\wh{f}(w)|(1+|w|^3)\diff w\;\ll\;\nu^{-7/2}.
\]
Therefore
\[
\int_{-\infty}^\infty|\wh{f}(w)| H(w) K(w)\diff w\;\ll\;k^2r^2(r/\log\alpha)^3\nu^{-7/2}e^{-c/2r\theta}.
\]
Finally we sum over $j\ll J=\log(r/3)/\log\alpha$, which gives
\[
\sum_{j}\sum_{\chi\ne\chi_0}\con{\chi}(\mc{C})W^\pm_j(\chi)\;\ll\;k^2r^6(\log\alpha)^{-4}\nu^{-7/2}e^{-c/2r\theta}.
\]
From \eqref{choice of J} we have $(\log\alpha)^{-4}\ll e^{2/9r\theta}$, which completes the proof.
\end{proof}

\appendix

\section{Approximating $L$-functions by finite Euler products}
\label{section euler products}

Let $L(s,f)$ be an entire $L$-function of degree $d\ge1$, where we think of $f$ as some interesting arithmetic object to which $L(s,f)$ is attached. It is natural to try to approximate $L(1,f)$ by a partial Euler product. This question and similar ones have been addressed by many works in the literature---see for instance \cite{granville2003distribution}, \cite{granville2005extreme}, \cite{gonek2007hybrid}, \cite{bui2007mean}, \cite{bui2008mean}, and \cite{gonek2012finite}, where such approximations are both developed and used for interesting arithmetic applications.

The main result in this section, Proposition \ref{proposition approximating L function by Euler product}, would follow from results in the cited works above (for $L(s,\chi_D)$, at least, from a slight modification of the results in \cite{bui2007mean}) after using an appropriate zero-density estimate for $L(s,f)$. However, we wish to give here a self-contained proof of the result that does not rely on any zero-density estimates. 

We assume that $L(s,f)$ is given by a Dirichlet series and Euler product,
\begin{align*}
L(s,f)&\;=\;\sum_{n\ge1}\lambda_f(n)n^{-s}\;=\;\prod_p L_p(s,f), \\
L_p(s,f) &\;=\; \prod_{1\le j\le d}(1-\alpha_j(p)p^{-s})^{-1},
\end{align*}
if $\sigma=\Re s>1$, where $|\alpha_j(p)|\le1$. Further we assume that $L(s,f)$ is entire and that it satisfies a functional equation of conductor $\Delta\ge3$, where 
\[
\gamma(s,f)\;=\;\pi^{-ds/2}\prod_{1\le j\le d}\Gamma\Big(\frac{s+\kappa_j}{2}\Big),\qquad \Re \kappa_j\ge0,
\]
is its gamma factor. This means (see Chapter 5 of \cite{iwaniec2004analytic})
\[
\Lambda(s,f)\;=\;\Delta^{s/2}\gamma(s,f)L(s,f)\;=\;\eps\Lambda(1-s,\con{f})
\]
where $\eps$ denotes the root number, $|\eps|=1$. 

Our goal is to approximate $L(1,f)$ by the finite product
\begin{equation}
\label{definition of E(x)}
E(x)\;=\;\prod_{p< x}L_p(1,f)
\end{equation}
when $x$ comparable to $\Delta$ in the logarithmic scale. Indeed, assuming the Riemann hypothesis for $L(s,f)$ shows that
\[
L(1,f)\;\sim\; E(x)\qquad\text{if }\;x\;\gg\;(\Delta\log\Delta)^2
\]
as $\Delta\to\infty$. By comparison, our result will be unconditional. We do not require any zero-density estimate for $L(s,f)$, only that it have a zero-free region of ``classical'' type; that is, we assume that Hypothesis $\text{H}(c)$ holds for $L(s,f)$.

\begin{proposition}
\label{proposition approximating L function by Euler product}
Suppose $L(s,f)$ is entire and that it satisfies Hypothesis $\text{H}(c)$. Let $\Delta=x^\theta$ for some $0<\theta<1$. Then
\begin{equation}
\label{approximating L function by Euler product}
L(1,f)\;=\;E(x)e^{\eta(x)},
\end{equation}
where
\begin{equation}
\label{explicit bound on eta}
\eta(x) \;\ll\; \exp\Big(\frac{-c}{3\theta}\Big) \;+\;\frac{1}{\log x}.
\end{equation}
The implied constant above depends on the degree $d$ and parameters $\kappa_1,\dots,\kappa_d$. 
\end{proposition}

\begin{remark}
A version of \eqref{explicit bound on eta} with explicit numerical constants is given in \cite{gaudet2023least}.
\end{remark}

In this article, we use the above result only for the quadratic Dirichlet $L$-function $L(s,\chi_D)$. In this case the result reads: assuming Hypothesis $\text{H}(c)$ holds for $L(s,\chi_D)$, then for $z>D$ we have
\begin{equation}
\label{approximation to L1chiD}
\prod_{p<z}\Big(1-\frac{\chi_D(p)}{p}\Big)\;=\;L(1,\chi_D)^{-1}\Big\{1+O\Big(\exp\Big(\frac{-c\log z}{3\log D}\Big)+\frac{1}{\log z}\Big)\Big\}.
\end{equation}
Before giving the proof of Proposition \ref{proposition approximating L function by Euler product}, we give a corollary that we need at other points in the work.

\begin{corollary}
\label{corollary summation for chiD}
Assume Hypothesis $\text{H}(c)$ holds for $L(s,\chi_D)$, and that $D=x^\theta$ where $0<\theta<1$.  Let $0<\alpha_1<\alpha_2<1$.  Then for $x$ sufficiently large we have
\begin{equation}
\label{summation for chiD}
\sum_{x^{\alpha_1}\le p<x^{\alpha_2}}\frac{\chi_D(p)}{p} \;\ll\; \exp\Big(\frac{-c\alpha_1}{3\theta}\Big).
\end{equation}
If $\phi$ is a smooth function supported on $[\alpha_1,\alpha_2]$, then we have
\begin{equation}
\label{smooth summation for chiD}
\sum_p\frac{\chi_D(p)}{p}\phi\Big(\frac{\log p}{\log x}\Big)\;\ll\;(\alpha_2-\alpha_1)^{-2}\exp\Big(\frac{-c\alpha_1}{\theta}\Big).
\end{equation}
\end{corollary}

\begin{proof}[Proof of Corollary \ref{corollary summation for chiD}]
For the function $L(s,\chi_D)$, we see from \eqref{approximating L function by Euler product} that
\[
\eta(x) \;=\; -\sum_{p\ge x}\log(1-\chi_D(p)p^{-1}) \;=\; \sum_{p\ge x}\frac{\chi_D(p)}{p}\;+\;O\Big(\frac{1}{x}\Big).
\]
Then \eqref{summation for chiD} follows from \eqref{explicit bound on eta}. The bound \eqref{smooth summation for chiD} follows from minor adaptations of the proof of Proposition \ref{proposition approximating L function by Euler product}.
\end{proof}

To prove Proposition \ref{proposition approximating L function by Euler product}, we begin with

\begin{lemma}
\label{zero density surrogate estimate}
Let $\rho=\beta+i\gamma$ denote a zero of $L(s,f)$. Then for all $x\ge3$,
\[
\sum_\rho(1+(1-\beta)\log x)^{-1}(1+(|\gamma|\log x)^2)^{-1}\;\le\; d+\frac{\theta}{2}+O(1/\log x),
\]
where the sum $\sum_\rho$ is taken over nontrivial zeros of $L(s,f)$, $\theta=\log\Delta/\log x$, and the implied constant depends only on the gamma factor parameters $\kappa_1,\dots,\kappa_d$. 
\end{lemma}

\begin{proof}[Proof of Lemma \ref{zero density surrogate estimate}]
For $\sigma=\Re s>1$ we have
\[
-\frac{L}{L}(s,f)\;=\;\sum_{n\ge1}\Lambda_f(n)n^{-s}
\]
with $\Lambda_f(n)$ supported on prime powers,
\begin{align*}
\Lambda_f(p^k)&\;=\;\sum_{1\le j\le d}\alpha_j(p)^k\log p, \\
\Lambda_f(p)&\;=\;(\alpha_1(p)+\cdots+\alpha_d(p))\log p=\lambda_f(p)\log p.
\end{align*}
Hence $|\Lambda_f(n)|\le d\Lambda(n)$ and
\begin{equation}
\label{log deriv L function bounded by zeta}
\Big|\frac{L'}{L}(\sigma,f)\Big|\;\le\; d\Big|\frac{\zeta'}{\zeta}(\sigma)\Big|\;=\;\frac{d}{\sigma-1}+O(1).
\end{equation}
On the other hand the Hadamard product yields
\[
-\frac{L'}{L}(s,f)\;=\;\frac{1}{2}\log\Delta+\frac{\gamma'}{\gamma}(s,f)-b-\sum_{\rho\ne0, 1}\Big(\frac{1}{s-\rho}+\frac{1}{\rho}\Big).
\]
Note the contribution of the trivial zeros $\rho$ to the above sum is $O(1)$. For the constant $b$, we have
\[
\Re b\;=\;-\sum_{\rho\ne0}\Re\frac{1}{\rho}.
\]
Therefore for $s=\sigma$ with $1<\sigma\le2$ we have
\begin{equation}
\label{from Hadamard product}
\Re\sum_{\rho}\frac{1}{\sigma-\rho}\;=\;\frac{1}{2}\log\Delta+\Re\frac{L'}{L}(\sigma,f)+O(1),
\end{equation}
where the $\rho$ summation now runs over all nontrivial zeros of $L(s,f)$. With $s=\sigma>1$ and $\rho=\beta+i\gamma$ we have
\begin{equation}
\label{bound for 1/sigma-rho}
\Re\frac{1}{\sigma-\rho}=\frac{\sigma-\beta}{(\sigma-\beta)^2+\gamma^2}\;\ge\; \frac{1}{\sigma-1}\Big(1+\frac{1-\beta}{\sigma-1}\Big)^{-1}\Big(1+\Big(\frac{|\gamma|}{\sigma-1}\Big)^2\Big)^{-1}.
\end{equation}
Thus now we choose $s=\sigma=1+1/\log x$ (where $x\ge 3$) and combine \eqref{log deriv L function bounded by zeta}, \eqref{from Hadamard product}, and \eqref{bound for 1/sigma-rho} to prove the lemma.
\end{proof}

\begin{proof}[Proof of Proposition \ref{proposition approximating L function by Euler product}]
We have
\[
L(1,f)\;=\;E(x)e^{\eta(x)},
\]
where $E(x)$ is defined in \eqref{definition of E(x)}, and
\[
\eta(x)\;=\;\log\Big(\prod_{p\ge x}L_p(1,f)\Big).
\]
We put
\[
K(x)\;=\;\sum_{n\ge x}\frac{\Lambda_f(n)}{n\log n}.
\]
Note that this sum converges by virtue of Hypothesis $\text{H}(c)$---it does not converge absolutely. We then check that
\begin{align*}
|K(x)-\eta(x)| \;&\le\; \sum_{p<x}\sum_{k\ge\frac{\log x}{\log p}}\frac{d}{kp^k}\;\le\; d\sum_{p<x}\frac{\log p}{\log x}\sum_{k\ge\frac{\log x}{\log p}}\frac{1}{p^k} \nonumber \\
&\le\; d\sum_{p<x}\frac{\log p}{\log x}\cdot\frac{2}{x}\;\ll\;\frac{1}{\log x},
\end{align*}
and therefore we have
\[
\eta(x)\;=\;K(x)\;+\;O\Big(\frac{1}{\log x}\Big),
\]
where the implied constant depends on $d$. To estimate $K(x)$, we put
\[
K_\phi(x)\;=\;\sum_{n\ge1}\phi\Big(\frac{\log n}{\log x}\Big)\frac{\Lambda_f(n)}{n\log n},
\]
where $\phi(u)$ is a smooth function satisfying $\phi(u)=0$ if $u\le1$, $0\le \phi(u)\le1$ if $1\le u\le1+\eps$, and $\phi(u)=1$ if $u\ge1+\eps$, with $0<\eps<1$. Then we have
\begin{align*}
|K(x)-K_\phi(x)|\;&\le\; d\sum_{x\le n\le x^{1+\eps}}\frac{\Lambda(n)}{n\log n}\;\le\;d\sum_{x\le p\le x^{1+\eps}}\frac{1}{p}\;+\;d\sum_{p\le x}\sum_{k\ge\max(2,\frac{\log x}{\log p})}\frac{1}{kp^k} \\
&\le\;d\log(1+\eps)\;+\;d\sum_{p\le x}\frac{\log p}{\log x}\sum_{k\ge2}\frac{1}{p^k}\;\ll\;\eps\;+\;\frac{1}{\log x},
\end{align*}
where again the implied constant depends on $d$. Thus now we have
\begin{equation}
\label{difference of K and Kphi}
\eta(x)\;=\;K_\phi(x)\;+\;O\Big(\eps+\frac{1}{\log x}\Big).
\end{equation}
To estimate $K_\phi(x)$, we apply the explicit formula for $L(s,f)$,
\begin{equation}
\label{explicit formula for Kf}
K_\phi(x)\;=\;-\sum_\rho \wt{\Phi}(\rho),
\end{equation}
where $\wt{\Phi}(s)$ is the Mellin transform of
\[
\Phi(y)\;=\;\phi\Big(\frac{\log y}{\log x}\Big)\frac{1}{y\log y}.
\]
We have
\begin{align}
\wt{\Phi}(s)&\;=\;\int_0^\infty \Phi(y)y^{s-1}\diff y\;=\; \int_{-\infty}^\infty \phi(u)x^{u(s-1)}u^{-1}\diff u \nonumber \\
\label{after partial integration three times}
&\;=\;(1+(1-s)\log x)^{-3}\int_{1}^\infty(e^uu^{-1}\phi(u))'''e^{-u}x^{u(s-1)}\diff u
\end{align}
by partial integration three times. We can choose $\phi(u)$ so that for $1\le u\le1+\eps$,
\[
(e^uu^{-1}\phi(u))'''\;\ll\; e^u\eps^{-3},
\]
and for $u\ge1+\eps$ we have
\[
(e^uu^{-1}\phi(u))'''\;\ll\;e^{u}.
\]
Using these bounds in \eqref{after partial integration three times}, we get
\begin{equation}
\label{first bound for phi}
\wt{\Phi}(s)\;\ll\;|1+(1-s)\log x|^{-3}\Big(\eps^{-2}+((1-\sigma)\log x)^{-1}\Big)x^{\sigma-1}.
\end{equation}
We also have the bound
\begin{equation}
\label{other bound for phi}
|\wt{\Phi}(s)|\;\ll\;\eps^{-2}|(1-s)\log x|^{-3},
\end{equation}
which is derived in a similar way, using the identity
\[
\wt{\Phi}(s)\;=\;((1-s)\log x)^{-3}\int_{1}^\infty(u^{-1}\phi(u))'''x^{u(s-1)}\diff u
\]
and the bound
\[
\int_1^\infty(u^{-1}\phi(u))'''\diff u\;\ll\;\eps^{-2}.
\]
Recall that the number of zeros $\rho=\beta+i\gamma$ with $0<\beta<1$ and $t<|\gamma|\le t+1$ is $O(\log\Delta(|t|+1))$ for all $t$. Using this with \eqref{other bound for phi}, we see that the contribution of $\rho=\beta+i\gamma$ with $|\gamma|>1$ to the explicit formula \eqref{explicit formula for Kf} is bounded by
\begin{equation}
\label{bound for high zeros}
\sum_{k\ge1}\sum_{k<|\gamma|\le k+1}\wt{\Phi}(\rho)\;\ll\;\eps^{-2}(\log\Delta)\sum_{k\ge1}k^{-2}(\log x)^{-3}\; \ll\;\eps^{-2}(\log x)^{-2}.
\end{equation}
The remaining zeros $\rho=\beta+i\gamma$ satisfy \eqref{assumption on zeros}, so for these zeros we have
\begin{equation}
\label{for the low zeros}
|x^{\rho-1}|\;\le\;\exp\Big(\frac{-c\log x}{\log \Delta}\Big)\;=\;\exp\Big(\frac{-c}{\theta}\Big),\quad\text{and}\quad (1-\beta)\log x\;\ge\;\frac{c}{\theta}.
\end{equation}
Noting that
\[
|1+(1-\rho)\log x|^{-3}\;\le\;(1+(1-\beta)\log x)^{-1}(1+(|\gamma|\log x)^2)^{-1},
\]
we apply Lemma \ref{zero density surrogate estimate} to see that the contribution of these zeros to \eqref{explicit formula for Kf} is
\begin{equation}
\label{bound for remaining zeros}
\sum_{\substack{\rho=\beta+i\gamma \\ |\gamma|\le1}}|\wt{\Phi}(\rho)|\;\ll\;\Big(d+\frac{\theta}{2}\Big)\Big(\eps^{-2}+\frac{\theta}{c}\Big)\exp\Big(\frac{-c}{\theta}\Big)
\end{equation}
by \eqref{first bound for phi} and \eqref{for the low zeros}. The bound above is larger than the one in \eqref{bound for high zeros}, and hence $K_\phi(x)$ is also bounded by the right-hand side of \eqref{bound for remaining zeros}. Combining this bound with \eqref{difference of K and Kphi} then shows
\[
\eta(x)\;\ll\; \eps+\eps^{-2}\exp\Big(\frac{-c}{\theta}\Big)+\frac{1}{\log x}.
\]
Finally we choose $\eps=\exp(-c/3\theta)$, which gives \eqref{explicit bound on eta}.
\end{proof}

\pagebreak


\printbibliography

\end{document}